\documentclass[a4paper,12pt]{amsart}

\usepackage{amsmath}
\usepackage{amssymb,amsbsy,amsmath,amsfonts,amssymb,amscd}
\usepackage{latexsym}
\usepackage{amsthm}
\usepackage{graphics}
\usepackage{color}
\usepackage{tikz}
 \usepackage{tikz-3dplot}
\usetikzlibrary{arrows}
\usepackage{tikz-cd}
\usepackage{pdfpages} 
\usepackage{longtable}
\usepackage{xy}
\usepackage{array}
\usepackage{color}
\usepackage{comment}

\input xy
\xyoption{all}

\newcommand\sC{{\mathcal C}}

\newcommand\sE{{\mathcal E}}
\newcommand\sA{{\mathcal A}}

\newcommand\sG{{\mathcal G}}

\newcommand\sU{{\mathcal U}}

\newcommand\sB{{\mathcal B}}

\newcommand\sN{{\mathcal N}}
\newcommand\sK{{\mathcal K}}

\newcommand\sH{{\mathcal H}}

\newcommand\sM{{\mathcal M}}

\usepackage{textcomp} 
               
\newcommand\om{\omega}
\newcommand\la{\lambda}
\newcommand\Lam{\Lambda}
\newcommand\al{\alpha}
\newcommand\be{\beta}
\newcommand\e{\epsilon}
\newcommand\s{\sigma}

\newcommand\Ga{\Gamma}

\newcommand\ga{\gamma}

\newcommand{\CC}{\ensuremath{\mathbb{C}}}
\newcommand{\RR}{\ensuremath{\mathbb{R}}}
\newcommand{\ZZ}{\ensuremath{\mathbb{Z}}}

\newcommand{\sS}{\ensuremath{\mathcal{S}}}

\newcommand{\hol}{\ensuremath{\mathcal{O}}}

\newcommand{\PP}{\ensuremath{\mathbb{P}}}

\newcommand{\FF}{\ensuremath{\mathbb{F}}}

\newcommand{\ra}{\ensuremath{\rightarrow}}

\def\eea{\end{eqnarray*}}
\def\bea{\begin{eqnarray*}}

\newcommand\dual{\mathrel{\raise3pt\hbox{$\underline{\mathrm{\thinspace d
\thinspace}}$}}}
\newcommand\qe{\ifhmode\unskip\nobreak\fi\quad $\Box$}       

\def\BOX{\hfill\lower.5\baselineskip\hbox{$\Box$}}

\newtheorem{theorem}{Theorem}
\newtheorem{theo}[theorem]{Theorem}
\newtheorem{remark}[theorem]{Remark}
\newenvironment{rem}{\begin{remark}\rm}{\end{remark}}
\newtheorem{question}[theorem]{Question}
\newtheorem{prop}[theorem]{Proposition}
\newtheorem{cor}[theorem]{Corollary}
\newtheorem{lemma}[theorem]{Lemma}
\newtheorem{example}[theorem]{Example}
\newenvironment{ex}{\begin{example}\rm}{\end{example}}

\theoremstyle{definition}
\newtheorem{defin}[theorem]{Definition}

\newenvironment{dedication}
        {\begin{quotation}\begin{center}\begin{em}}
        {\par\end{em}\end{center}\end{quotation}}

\usepackage{hyperref}

\makeatletter
\def\tagform@#1{\maketag@@@{\ignorespaces#1\unskip\@@italiccorr}}
\makeatother

\newcolumntype{H}{@{}>{\lrbox0}l<{\endlrbox}} 

\begin{document}

\title[Kummer surfaces, selfduality]{Kummer quartic  surfaces, strict self-duality, and  more}
\author{Fabrizio Catanese}
\address{Lehrstuhl Mathematik VIII, 
 Mathematisches Institut der Universit\"{a}t
Bayreuth, NW II\\ Universit\"{a}tsstr. 30,
95447 Bayreuth, Germany \\ and Korea Institute for Advanced Study, Hoegiro 87, Seoul, 
133--722.}
\email{Fabrizio.Catanese@uni-bayreuth.de}

\thanks{
The author acknowledges support of the ERC 2013 Advanced Research Grant - 340258 - TADMICAMT}

\maketitle

\begin{dedication}
Dedicated to Ciro, the `Prince of Stromboli',  on the occasion of his 70-th  
 birthday.
\end{dedication}

\begin{abstract}
In this paper we first show that each Kummer quartic surface (a quartic surface $X$ with 16 singular points) is, in canonical coordinates,
 equal to its dual surface, and that the Gauss map induces a fixpoint free involution $\ga$ on the minimal resolution
 $S$ of $X$. Then we study the corresponding Enriques surfaces $S/ \ga$. 
 
 We also describe in detail the remarkable
 properties of the  most symmetric Kummer quartic,
 which we call the Cefal\'u quartic. We also investigate the Kummer quartic surfaces
 whose associated Abelian surface is isogenous to a product of elliptic curves  through an isogeny  with kernel $(\ZZ/2)^2$, and show the existence of
 polarized nodal K3 surfaces $X$ of any degree $d=2k$ with the maximal number of nodes, such that  $X$
 and its nodes are defined over $\RR$.
 
 We take then as  parameter space  for Kummer quartics an open set in $\PP^3$,
 parametrizing nondegenerate $(16_6, 16_6)$-configurations, and compare with other parameter spaces. 
 
 We also extend to positive characteristic some results which were previously known over $\CC$.
 
 We end   with a section devoted to   remarks on normal cubic surfaces, and providing  some other  examples of strictly selfdual hypersurfaces.

\end{abstract}

\tableofcontents

\setcounter{section}{0}

\section{Introduction} 

This paper originated from some simple examples I gave in  a course held in Bayreuth in summer 2018, and in  a series of Lectures held in November 2018 in Udine, on the topic of  surfaces in $\PP^3$ and their singularities.

Given an irreducible surface $X$ of degree $d$ in $\PP^3$ which is normal, that is, with only   finitely many singular points, one
can give an explicit upper bound $\mu' (d)$ for the number of its singular points (for instance, over $\CC$, one can use the birational map 
of $X$ to its dual surface $X^{\vee}$ to obtain a crude estimate). 

 Letting $\mu(d)$ be the maximal number of singular points of a normal surface of degree $d$, the case of $d=1,2$ being trivial ($\mu (1)= 0, \mu (2)=1$), the first interesting cases are for $d= 3,4$: $\mu (3)= 4, \mu (4)=16$  if char(K) $ \neq 2$ \footnote{ In
 $ char(K) = 2$, using also a remark by the referee on  the discriminant   of supersingular K3 surfaces we can prove that    $\mu(4) \leq  20$.}.

For $d=3$ (see  propositions \ref{nodalcubic} and \ref{monoid}), a normal cubic surface $X$ can have at most 4 singular points, no three of them can be collinear, 
and if it does have 4 singular points, these are linearly independent, hence $X$  is projectively equivalent to the 
so-called Cayley cubic, first apparently found by Schl\"afli, see \cite{schlaefli}, \cite{cremona},  \cite{cayley}.

The Cayley cubic has the simple equation

$$X : = \{ x : = (x_0, x_1, x_2, x_3) | \s_3(x ): =  \sum_i \frac{1}{x_i} x_0 x_1 x_2 x_3 = 0\} $$ 
Here $\s_3$ is the third elementary symmetric function (the four singular points are the 4 coordinate points).

A normal  quartic surface can have,   if char(K) $ \neq 2$, at most 16 singular points: over $\CC$, one can use in general the birational map 
of $X$ to its dual surface $X^{\vee}$ to obtain  the crude estimate that
the dual surface $X^{\vee}$ of a normal surface $X$ with $\nu$ singular points has degree $ \leq d (d-1)^2 - 2 \nu$. 
 By    biduality  the  degree of  $X^{\vee}$  is at least $3$ if $ d \geq 3$.

Hence, for $d=3$, we get $\nu \leq 4$, as  already mentioned, while for $d=4$ $\nu \leq 16$,
equality holding if and only if $X^{\vee}$ is also a quartic surface.

 More generally,  if char(K) $ \neq 2$, and $X$ is a normal quartic surface, by proposition \ref{monoid} it has at most $7$ singular points if it has a triple point,
else  it suffices to project from a double  point of the quartic  to the plane, and to use the bound for the number of singular points for a plane curve of degree $6$, which equals $15$,  to establish $\nu \leq 16$.

Quartics with $16$ singular points  (char(K) $ \neq 2$) have necessarily nodes as singularities, and they are the so called 
 Kummer surfaces \cite{kummer} (the first examples were found by Fresnel, 1822).

There is a long history of research on Kummer quartic surfaces, for instance it is  
 well known  that if $d=4, \nu = 16$, then $X$ is the quotient of a
principally polarized  Abelian surface
$A$ by the group $\{\pm 1\}$, embedded (for $K = \CC$) by the theta functions of second order on $A$.

The simplest example of such a surface is the following 

{\bf Main example:}
 The Duke of Cefal\'u quartic surface is the 16-nodal quartic $X$ with 
  equation $$ s_1(z_i^2)^2 - 3   s_2(z_i^2) : = (\sum_1^4 z_i^2)^2 - 3 \sum_1^4 z_i^4 = 0 \  \Leftrightarrow \
  \s_2( z_i^2) - s_2(z_i^2) = 0.$$
  
Here $\s_i$ is the i-th elementary symmetric function, while $s_i$ is the i-th Newton function (sum of i-th powers,
in particular $s_2(z_i^2) = s_4 (z_i)$).

The Duke of Cefal\'u K3 surface shall denote  the minimal resolution $S$ of $X$, and  
for  simplicity  we shall also call $X$   the Cefal\'u  quartic.  Here are some remarkable properties of
this surface.

\begin{theo}\label{Cefalu'}
The Cefal\'u  quartic surface $$ X : = \{  \s_2( z_i^2) - s_2(z_i^2) =0\}= \{ (\sum_1^4 z_i^2)^2 - 3 \sum_1^4 z_i^4=0\} \subset \PP^3_K$$ ($K$ an algebraically closed field of characteristic $\neq 2,3$) enjoys the following properties:
\begin{enumerate}
\item
The group $\sG$ of projectivities which leave $X$ invariant contains the semidirect product
$$ G : =  (\ZZ / 2)^3  \rtimes \mathfrak S_4 ,$$ central quotient of the canonical  semidirect product $G'$
$$ G : =  G' / Z(G') , \ G'  : = (\ZZ / 2)^4  \rtimes \mathfrak S_4,$$
and $\sG = G$  if   the  characteristic of $K$ is $\neq 5$, else $|\sG| = 10 |G| = 1920$.
\item
The singular points of $X$ are $16$, and they are  the $G$-orbit $\sN$ of the point $(1,1,1,0)$, and 
also its $\sK$- orbit, where $\sK \cong (\ZZ/2)^4$ is the subgroup  generated by  double transpositions
and by diagonal matrices with  entries $\pm 1$ and determinant $1$.
\item
$X$ is strictly self-dual in the strong sense that $X^{\vee} = X$.
\item
Via the standard scalar product the  set $\sN$ yields a set $\sN'$ of $16$ planes, such that $(\sN, \sN')$
form  a $(16_6, 16_6)$ configuration. 
\item
The Gauss map gives an automorphism $\ga$ of  the  minimal resolution $S : = \tilde{X}$ of $X$, which has order $2$ and
which centralizes $G$.
\item
The  group $ G \times \ZZ/ 2$ generated by $G$ and $\ga$ contains a subgroup $G^s $ of index $2$ and order $192$,  $ G^s  \cong  (\ZZ / 2)^3  \rtimes \mathfrak S_4 ,$
of automorphisms of $S$ acting symplectically (acting trivially on the holomorphic two form of $S$).
 The full group of symplectic automorphisms of $S$ has order $384$ and is a semidirect product 
$(\ZZ / 2)^4  \rtimes \mathfrak S_4 .$
\item
$\ga$ acts without fixed points, hence $S / \ga$ is an Enriques surface, which contains  a configuration of $16$ $(-2)$-curves,
which we describe via the  called Kummer-Enriques graph, and
containing several sets of cardinality $4$ of disjoint such curves, but none of cardinality $5$
(in particular  $S / \ga$ is the minimal resolution of several  $4$-nodal Enriques surfaces).
\item
The group of automorphisms of $S$ is larger than $G \times \ZZ/ 2$  and indeed infinite.
\item
The surface $X$ is the Kummer surface of an Abelian surface isogenous to the product of two Fermat 
(equianharmonic) elliptic curves.
\item
The Cefal\'u  quartic surface $$ X : = \{  \s_2( z_i^2) - s_2(z_i^2) =0\} \subset \PP^3_\ZZ$$
defines a Kummer surface over $\sS : = Spec (\ZZ(\frac{1}{2}, \frac{1}{3}))$.
\end{enumerate}
\end{theo}

Do some   of these properties carry over to all Kummer quartic surfaces?
Yes, (ii), (iii), (iv), (v), (vii), (viii) do,  as we shall see in this paper; the other properties  are very special of this surface or extend only in part or for special Kummer quartics.

There are several comments to be made.
\begin{itemize}
\item
Concerning (i), we get a group of symplectic automorphisms of order $192$, which is an index two  subgroup of
the group of symplectic automorphisms of the Fermat quartic surface (see Mukai's paper \cite{mukai}, item 5) on page 184,
and line 8 on page 191). \footnote{ The vague question here is whether    the Cefal\'u quartic and the Fermat quartic 
are somehow related,  `mirror' to each other?  
Motivation comes from  the pencil containing them both  and consisting precisely of the quartics having  $G$-symmetry:
$$ \la_1 s_1(z_i^2)^2  + \la_2   s_2(z_i^2) =0  \Leftrightarrow \
 \mu_1 \s_2( z_i^2) - \mu_2 s_2(z_i^2) = 0.$$ 
 The Cefal\'u quartic is the only Kummer quartic in the pencil, which has five singular elements, for $\la_1=1, \la_2= 0,-1, -2,-3, -4$.}
 We show in proposition \ref{fermat} that indeed the full group of symplectic automorphisms of the Cefal\'u K3 surface $S$ is equal to the group of symplectic automorphisms of the Fermat quartic surface.
\item
 Concerning (vii), the existence of (other)  fixed point free involutions on Kummer surfaces (thereby leading to new constructions of Enriques surfaces) was first shown by Hutchinson, who used Cremona involutions based at the so-called G\"opel tetrads.
\cite{hutchinson}.
\item
Concerning (viii), for our assertion, leading to question 18, we use 
the Klein involution $\iota$ induced from projection from one node. 
  Our argument allows us to prove
Theorem \ref{infinite} asserting that for each Kummer quartic surface the group of automorphisms of $S$
is infinite, in characteristic $\neq 2$. Our result confirms the unpublished  result of Jong Hae Keum (over $\CC$), cited in  \cite{keum-kondo},  and based on the combination of the results of \cite{kondothesis} and \cite{keum1}.

Over $\CC$, since for the Cefal\'u surface $S$ $Pic(S) $ has rank $20$, 
the fact that $Aut(S)$ is infinite follows from the old result of Shioda and Inose \cite{shioda-inose}. 
\item
Also 
 Hutchinson asserts that the group $Aut(S)$ of a Kummer surface
 is infinite (at least for general choice of the Kummer surface, compare the argument given on page 214 of 
\cite{hutchinson}). Hutchinson  uses the   involutions mentioned above, but we show that these do not exist as claimed by Hutchinson
on the Cefal\'u surface, see remark \ref{Tetrahedron};  more generally they do not exist for Segre-type Kummer surfaces, because,  if some Thetanull vanishes, see remark \ref{thetanull}, then the G\"opel tetrads become linearly dependent.  
\item
Concerning (iii),   we observe that   in all articles and textbooks
 one finds the weaker statement (inspired by the great article by Klein \cite{klein} treating quadratic line complexes) that a Kummer quartic is projectively equivalent to its dual surface (see for instance \cite{g-h}, page 784,
  \cite{dorrego} cor. 4.27, page 95, \cite{dolgachev}  theorem 10.3.19, and remark 10.3.20). To clarify the issue, we give the following 
\end{itemize}
\begin{defin}
A projective variety $X \subset \PP^n$ is strictly self-dual if there are coordinates such that $X^{\vee} = X$.

$X$  is said to be weakly self-dual if $X$ is projectively equivalent to $X^{\vee}$.
\end{defin}

To focus on the difference between the two notions, it suffices to observe that  $X$ is strictly self-dual
if and only if there is a projectivity sending $X$ to $X^{\vee}$ whose matrix $A$ is of the form $^{t }B B$,
equivalently, such $A$ is symmetric.

Up to now the main series of examples of self-dual varieties (\cite{cat-slava} , \cite{popov}, \cite{popov-tevelev}) yielded strongly self-dual varieties.

The question whether there are weakly self-dual varieties which are not strongly self-dual remains (as far as we know)  open, since
for Kummer quartics we prove the analogue of (iii) and (vii) above:

\begin{theo}\label{gauss}
(1) All Kummer quartics $X \subset \PP^3_K$, $K$ an algebraically closed field of characteristic $\neq2,3$  are strongly self-dual
\footnote{probably a similar computation shows the result also in $char=3$}, and 

(2) the Gauss map produces  an automorphism $\ga$, of order $2$, of the minimal resolution $S : = \tilde{X}$ of $X$,
which acts without fixpoints, so that  $S / \ga$ is an Enriques surface.

\end{theo}

We give two proofs for the above statement, a short transcendental proof over $\CC$ using theta functions, and 
then  an algebraic
one (valid in characteristic $\neq 2$) which  relies on earlier results, especially the ones
by  Gonzalez-Dorrego (\cite{dorrego}).

Observe that (2) provides, over $\CC$,  a simpler proof of an important result of Keum \cite{keum1}, that every Kummer surface is an
unramified covering of an Enriques surface; and it extends the result also to fields $K$ of positive characteristic. 

Many explicit equations have been given for Kummer quartics (I refer for this  to the beautiful survey by Dolgachev \cite{igornams},
which appeared just before this paper was be finished), but here we use new free parameters for Kummer quartics.

These are based on the following theorem, essentially due to Gonzalez-Dorrego, extending (iv) above to all Kummer quartic surfaces
over an algebraically closed field $K$  of characteristic $\neq2$.

We consider for this purpose the action on $\PP^3$ of the (already mentioned) group 
$$ \sK : = (\ZZ/2)^4 = (\ZZ/2)^2 \oplus (\ZZ/2)^2$$
such that the first summand acts via the group of double transpositions, the second summand 
through $ z_i \mapsto \e_i z_i$, for $ \e_i = \pm 1, \   \e_1 = \e_1  \e_2 \e_3 \e_4 = 1$,
and we recall that a nondegenerate $(16_6, 16_6)$ configuration in $\PP^3$ is a configuration $\sN, \sN'$ of $16$ points
and $16$ planes such that each point in $\sN$ belongs to $6$ planes of the set $\sN'$ , and each plane in $\sN'$ contains 
exactly $6$ points of the set $\sN$, and nondegenerate means that two planes of $\sN'$ contain exactly $2$
common points of $\sN$.

\begin{theo}\label{166}
Every nondegenerate $(16_6, 16_6)$ configuration in $\PP^3_K$, $K$ algebraically closed with $char (K) \neq 2$ is projectively equivalent to the configuration $(\sN, \sN')$
where $\sN'$ is the set of planes orthogonal to the elements of  $\sN$, and $\sN$ is  the $ \sK  = (\ZZ/2)^4$ orbit 
 of a point
$(a_1, a_2, a_3, a_4) \in \PP^3$ such that: 

(I) no two coordinates are equal to zero, $\sum_i a_i^2 \neq 0$, 
$$ (II) \ a_1 a_2 \pm  a_3 a_4  \neq 0, a_1 a_3 \pm  a_2 a_4  \neq 0, a_1 a_4 \pm  a_2 a_3 \neq 0,$$
$$ (III) \ a_1^2 + a_2^2 \neq   a_3^2 +  a_4^2, a_1^2 +  a_3^2 \neq  a_2^2 +  a_4^2  , a_1^2 +  a_4^2  \neq  a_2^2 +  a_3 ^2.$$ 

For each such configuration $\sN$, which is strictly self-dual, there is a unique Kummer quartic $X$ having $\sN$
as singular set, and  all Kummer quartics arise in this way.

The equation of $X$ has the  Hudson normal form 
$$  \al_0 (\sum_i z_i^4)  + 2  \al_{01} (z_1^2 z_2^2 + z_3^2 z_4^2) + 2 \al_{10} (z_1^2 z_3^2 + z_2^2 z_4^2) + 2 \al_{11} (z_1^2 z_4^2 + z_2^2 z_3^2) + 4  \be z_1 z_2 z_3 z_4 = 0$$
 where, setting 
$b_i : = a_i^2, b := a_1 a_2 a_3 a_4$,  the vector of coefficients $$v : = ^t (\al_0, \al_{01}, \al_{10}, \al_{11},\be)$$ is given, for $b \neq 0$,
 by  the solution of the system of linear equations
$$ B v = 0$$
where $B$ is the matrix
	\[
		B : = \left( 
			\begin{array}{ccccc}
				b_1^2 & b_1 b_2 & b_1 b_3 & b_1 b_4 & b\\
				b_2^2 & b_2 b_1 & b_2 b_4 & b_2 b_3 & b\\
				b_3^2 & b_3 b_4 & b_3 b_1 & b_3 b_2 & b\\
				b_4^2 & b_4 b_3 & b_4 b_2 &  b_4 b_1& b\\
				
			\end{array}
		\right).
	\]
 The case  $b =0$  is governed by other equations and, assuming without loss of generality that $a_1 = 0,$ we obtain  $\be = 0$, and we get the coefficients:
\begin{align*}
		 \al_0 &= 2 b_2 b_3 b_4 \\
		\al_{01} &=  b_2 (b_2^2 - b_3^2 - b_4^2) \\
		\al_{10} &= b_3 ( b_3^2 - b_4^2 - b_2^2) \\
		\al_{11} &= b_4 ( b_4^2 - b_2^2 -  b_3^2 ) \\
	\end{align*}

\end{theo}

 For $b \neq 0$, by the theorem of Rouch\'e-Capelli, the coefficients of $v$ are given by the determinants of the $4\times 4$ minors
 of the matrix $B$, which we omit to spell out in detail.

 The case $b=0$ is quite interesting and indeed related to how we found the Cefal\'u surface as the most symmetric among the Segre-type
 Kummer quartics . These are the quartics whose equation is of the form $ Q( z_1^2 , z_2^2 , z_3^2 , z_4^2)$
 where $Q(x)=0$ is a smooth quadric and the four planes $x_i = 0$ are tangent to it at a point not lying on the edges of
 the tetrahedron $T : = \{x_1 x_2 x_3 x_4 = 0\}$.
 
 Observe in fact that the dual quadric $Q^*(y) =0$ has a matrix with diagonal entries equal to $0$,
 hence the most symmetric solution is $\s_2 (y) = 0$, leading to the Cefal\'u quartic.
 
 We devote a section to the geometry of these special Kummer quartics, and show that these are  the ones whose associated Abelian surface $A$
 is isogenous to a product of elliptic curves $A_1 \times A_2$  through an isogeny with kernel group $(\ZZ/2)^2$. And we use them in order to show a result which is used in \cite{nodalsurfaces}
 to determine the components of the variety of nodal K3 surfaces.

 \begin{theo}\label{real}
For each degree $ d = 4m$ there is a nodal K3 surface $X'$ of degree $d$ with $16$ nodes which is defined over $\RR$, and such that
all its singular points are defined over $\RR$.

 Similarly,  for each degree $ d = 4m - 2$ there is a nodal K3 surface $X''$ of degree $d$ with $15$ nodes which is defined over $\RR$, and such that
all its singular points are defined over $\RR$.

The same results holds replacing $\RR$ by an algebraically closed field $K$ of characteristic $\neq2$.

\end{theo}
 
\begin{rem}
As shown for instance  in \cite{nodalsurfaces} if the degree $d$ is not divisible by $4$, then  $15$ is the maximal number of nodes.
\end{rem}

Later on,  we compare the equations provided by theorem \ref{166} with other representations of Kummer surfaces,
for instance the one given by  Coble \cite{coble} and later van der Geer \cite{vanderGeer} through 
 the Siegel quartic modular threefold, which is the dual of the 10 nodal Segre cubic threefold $\sS$. 
 
 Projective duality is then shown to provide the relation between  two different representations of Kummer surfaces,
 
 (1) the first one as discriminants of the projection to $\PP^3$  of the  Segre cubic threefold $\sS$ with centre a smooth point $x  \in \sS$,
 not lying  in the $15$ planes
 contained in $\sS$,
 
(2)  the second one  as intersections of $\sS^{\vee}$ with a tangent hyperplane.

\bigskip

In the end, we consider the Enriques surface $ Z: = S/ \ga$ obtained from a Kummer K3 surface,  study 
the set of $16$ $(-2)$-curves on it, and the associated graph (we call it the Enriques-Kummer graph), 
 which  turns out to be  associated to a triangulation of the 2-torus.
In this way we are able to show that  $Z$ contracts in several ways to a $4$-nodal Enriques surface.

In a final section we describe  simple illuminating examples, concerning self dual hypersurfaces and simple proofs of results on normal cubic surfaces\footnote{ 
these formed an important chapter  in my 
lectures.}.

A question we could not yet answer is whether the full group of birational automorphisms of the Cefal\'u quartic is  generated
(in characteristic $\neq 2,3, 5$) by 
the subgroup $G$ of projectivities, and by the involutions $\ga$ and $\iota$ ( $\iota$ is  induced by the projection from one node)  (see \cite{kondo} and \cite{keum2} for results 
in the case of  generic Kummer surfaces, that is, with Picard number $17$,  and \cite{keum-kondo} for Kummer surfaces of the product of two elliptic curves);  
as already mentioned, we show at any rate in proposition \ref{fermat}  that  the group $G^s $ of order $192$ is smaller 
than the group $Aut(S)^s$ of  symplectic automorphisms of $S$, which has order $384$ and is isomorphic to the group
of symplectic projectivities of the Fermat quartic surface.

Since the group  $Aut(S)^s$ of symplectic automorphisms of $S$ is a normal subgroup of  $Aut(S)$,
together with the involution $\ga$ it generates a finite subgroup of $Aut(S)$ of cardinality $768= 384 \cdot 2$
(and we can play the same game  with $\iota$).
Keum and Kondo, in Theorem  5.6 of \cite{keum-kondo},  exhibit, for the Kummer surface associated to the product of two Fermat curves, a
finite group of order $1152= 384 \cdot 3$. It seems  interesting to relate  automorphism groups of Kummer surfaces associated
to isogenous Abelian surfaces.

\subsection{Notation} For a point in projective space, we shall freely use two  notations: the vector notation $(a_1, \dots, a_{n+1})$,
and the notation $[a_1, \dots, a_{n+1}]$, which denotes the equivalence class of the above vector (the second notation  is meant to 
point out that in some  cases we
 do not have an equality of vectors, but only an equality of points of projective space).

\section{Proof of Theorem \ref{gauss} over $\CC$ via Theta functions.}

As a preliminary observation, we observe that, as proven by Nikulin \cite{nikulin-kum} (see also \cite{nodalsurfaces} for another argument),
every quartic with exactly $16$ singular points is the Kummer surface $K(A)$ associated to a principally polarized Abelian surface $A$ with smooth Theta divisor.

Hence we consider these Kummer varieties $K(A)$  in general,  under the assumption that $(A, \Theta)$
is not a product of polarized varieties 
$$(A, \Theta) \neq  (A_1, \Theta_1) \times (A_2, \Theta_2).$$

Let therefore $A$ be  a principally polarized complex Abelian variety  
$$ A = \CC^g / \Ga ,\  \Ga : = \ZZ^g \oplus \tau \ZZ^g ,$$ 
where $\tau $ is a $(g \times g)$ matrix in the Siegel upper-halfspace $$ \sH_g : = \{ \tau | \tau = \  ^t \tau , Im (\tau) > 0 \}.$$

The Riemann Theta function 

$$\theta (z, \tau) : = \sum _{m \in \ZZ^g}  {\bf e} (^t m z + \frac{1}{2} \ ^t m\tau m), \ z \in \CC^g, \tau \in  \sH_g, \ {\bf e}(x) : = exp (2 \pi i x),$$
defines a divisor $\Theta$ such that the  line bundle $\hol_A(\Theta)$ has a space of  sections spanned by $\theta$.

The space $H^0(\hol_A(2\Theta))$ embeds  (see \cite{ohbuchi}) the Kummer variety
$$K (A) : = A / \pm 1$$ in the projective space $\PP^{2^g-1}$, and 
\begin{defin}
The canonical basis of the space $H^0(\hol_A(2\Theta))$ is given by the sections 
$$ \theta_{\mu} (z, \tau) : = \theta[a,0] (2 z , 2 \tau) ,$$ for $a = \frac{1}{2} \mu$,
$\mu \in \ZZ^g / 2 \ZZ^g$, and where the theta function with rational characteristics $a,b$ is defined as:
$$\theta[a,b] (w ,  \tau') : =  \sum _{p \in \ZZ^g}  {\bf e} \left( \frac{1}{2} \ ^t (p+a) \tau' (p+a) + ^t (p+a) (w + b)\right).$$
\end{defin}

The following is the main formula for self-duality of Kummer surfaces:

 \begin{prop}\label{Lefschetz}
 For each $ u \in \CC^g$ the product $\theta (z + u ) \theta (z - u )$ defines a section $\psi_u \in H^0(\hol_A(2\Theta))$ with
 $$\psi_u = \sum_{\mu \in \ZZ^g / 2 \ZZ^g} \theta_{\mu} (u, \tau) \theta_{\mu} (z, \tau).$$
 
 \end{prop}
 \begin{proof}
 $$\psi_u : = \theta (z + u ) \theta (z - u )= \sum_{m,p \in \ZZ^g} {\bf e} \left( ^t m (z+u)  + \frac{1}{2} \ ^t m\tau m + ^t p (z-u)  +
  \frac{1}{2} \ ^t p \tau p\right)  =$$
 $$ = \sum_{m,p \in \ZZ^g} {\bf e} \left( ^t (m+p) (z)  + ^t (m-p) (u)  + \frac{1}{2} \ ^t (m + p) \tau (m+p)   - 
   \ ^t m \tau p\right)  .$$
   
   We set now:
   $$ m' : =  m + p = : 2 (M + \frac{\mu}{2}), \  M  \in \ZZ^g, \ \mu \in \ZZ^g / 2 \ZZ^g; p' : = (m-p) = : 2 M ' + \mu, \ M' := M - p.$$

 Then we can rewrite $\psi_u$ as:
$$  \sum_{\mu \in   \ZZ^g / 2 \ZZ^g}  \sum_{M \in \ZZ^g} {\bf e} \left( 2 \ ^t (M + \frac{\mu}{2}) z + \ ^t (M + \frac{\mu}{2}) \tau (M + \frac{\mu}{2})\right) \cdot$$
$$\cdot  \sum_{p' \in \ZZ^g}  {\bf e} \left(  ^t p' u + \ ^t (M + \frac{\mu}{2}) \tau (M + \frac{\mu}{2}) - ^t m \tau p\right) =$$
$$ = \sum_{\mu} \theta_{\mu}(z, \tau) \sum_{p' \in \ZZ^g}  {\bf e} \left(   2 \ ^t (M' + \frac{\mu}{2}) u + \ ^t (M' + \frac{\mu}{2}) \tau (M' + \frac{\mu}{2}) \right) =$$
$$ =  \sum_{\mu \in \ZZ^g / 2 \ZZ^g}  \theta_{\mu} (z, \tau) \theta_{\mu} (u, \tau).$$
 
 \end{proof}
 
 \begin{cor}
 Consider the Kummer variety $ K(A)$ of a principally polarized Abelian variety  with irreducible Theta divisor, embedded in 
 $\PP^{2^g-1}$ by the canonical basis 
 for $H^0(\hol_A(2\Theta))$. 
 
 Then $ K(A)$ is contained in the dual variety of $K(A)$.
 
  In particular, for
 dimension $g=2$, we have self duality
 $$K(A) = K(A) ^{\vee}.$$
 \end{cor}
 \begin{proof}
 By proposition \ref{Lefschetz} we have that to each $u$ is associated an element $\psi_u  \in H^0(\hol_A(2\Theta))$,
 to which  there corresponds a  Hyperplane section $H_u$ of $K(A)$ whose pull-back is the divisor $ (\Theta + u) + (\Theta - u)$.
 
 Since the two divisors  $ (\Theta + u) $ and $ (\Theta - u)$ intersect in some pair of  points $z$ and $-z$ (notice that all sections are
 even functions), we obtain that $H_u$ is singular, and its singular points are, for general $u$, not points of $2$-torsion in $A$.
 
 For $g=2$ we conclude since $K(A) ^{\vee}$ has dimension at most $2$ and is irreducible.
 
  \end{proof}
  
  \begin{theo}\label{free}
  Let $K(A)$ be the Kummer surface of a principally polarized Abelian surface ($g=2$),  with irreducible Theta divisor, embedded in $\PP^3$ 
  by the canonical basis.
  
  Then the Gauss map $\ga : K(A) \dashrightarrow K(A)$ induces  a fixed-point free involution on the minimal resolution of $K(A)$.
  
  $\ga$ blows up the node $P_{\eta} $ of $K(A)$,  image point of a $2$-torsion point $\eta$, to the conic $C_{\eta} $ image
  of $ \Theta + \eta$, such that $2 C_{\eta} $ is a plane section $H_{\eta} $  of $K(A)$.
  \end{theo}
 \begin{proof}
In this case, $\Theta$ being an irreducibe divisor with $\Theta^2=2$, $ (\Theta + u) \cap (\Theta - u)$
consists either of two distinct points $z$ and $-z$ or of a point $z$ of $2$-torsion counted with multiplicity $2$,
unless $ (\Theta + u) = (\Theta - u) \Leftrightarrow (\Theta + 2 u) = \Theta .$

The last case, since $\Theta$ is aperiodic, can only happen if  $2 u=0$.

Let us argue first for $ 2u \neq 0$.

We have  $\ga(\pm z) =\pm  u \in K(A)$, and if we had a fixed point on $K(A)$, then we would have $ u \in (\Theta + u) \cap (\Theta - u)$,
hence $\theta(0) = 0$; since $\theta$ is even, it follows that $\theta$ vanishes of order $2$ at $0$, contradicting the
fact that the divisor $\Theta$ is a smooth curve of genus $g=2$.

If instead $u$ is a $2$-torsion point $\eta$, then there is a plane $H_{\eta}$ which cuts the curve $C_{\eta}$ with multiplicity two,
so that the Gauss map contracts the conic image of $C_{\eta}$ to the point $\eta$, by proposition \ref{Lefschetz}.

Te rest of the proof is identical, since $\eta \in C_{\eta} \rightarrow 0 \in \Theta$ is again a contradiction.

$\ga$ is an involution ($\ga^2 = Id$) since the Gauss map of the dual variety is the rational  inverse 
of the original variety.

  \end{proof}
  
  \begin{rem}
 Many authors use  the weaker assertion that a Kummer quartic $X$ is projectively equivalent to its dual surface in
 order to  obtain an involution
  on the minimal resolution $S$,  called a switch, which exchanges the $(-2)$-curves $E_i$ which are the blow up of the nodes
 with the $16$ tropes $E'_i$,  the strict transforms of the conics $C_{\eta}$ (see for instance \cite{klein}, \cite{kondo}). They  compose the Gauss map with a projectivity $B$   giving  an isomorphism of
 $X^{\vee}$ with $X$. Denote by $f$ the birational automorphism thus obtained; then the square $f\circ f$ of $f$ yields an automorphism of $S$ leaving invariant the set of  all the nodal curves and the set of tropes; since the pull back via $f$ of $H$ equals $ 3H - \sum_1^{16} E_i$,
 the pull back via $f\circ f$ of $H$ equals $$ 3 ( 3H - \sum_1^{16}  E_i ) - \sum_1^{16}  E'_i  \equiv H,$$
because  $$ 2 \sum_1^{16}  E'_i \equiv 16 H - 6 \sum_1^{16}  E_i \Rightarrow   \sum_1^{16}  E'_i \equiv 8 H - 3 \sum_1^{16}  E_i.$$
 
Since $f\circ f$ sends $ H \mapsto H$,  
 the square $f\circ f$ is  a projectivity of $X$.  What is incorrect is the  claim of several authors that this square is the identity.
 
In fact, we have shown for the Cefal\'u quartic that $\ga^2 = 1$ and that $\ga$ centralizes the whole group $G$.
 If we take now as $B$ a projectivity
$ g \in G$, we have that $ \ga g \ga g = \ga^2 g^2 = g^2 \neq Id$ if $g$ does not have order $2$
(and such a $g$ exists  for the Cefal\'u quartic).
  
  \end{rem}
 
 We end this section observing that the image of a $2$-torsion point $\eta$ easily determines all others,
 essentially because (\cite{igusa}, \cite{mumfordtheta}) $H^0(\hol_A(2\Theta))$ is the Stone-von Neumann representation of
 the Heisenberg group $H(G)$, central extension by $\CC^*$ of 
 $$ G \times G^* \cong (\ZZ/ 2\ZZ)^g \times (\tau \ZZ^g / 2 \tau \ZZ^g).$$
 This representation (which is studied and used in \cite{caci}) is the space of functions on $G$, on which $G$ acts by translations, while the dual group of characters $G^* = Hom (G, \CC^*)$
 acts by multiplication with the given character.
  
  In simpler words (where we look at the associated action on projective space)  recalling that
 
 $$ \theta_{\mu} (z, \tau) : = \theta[a,0] (2 z , 2 \tau) =  \sum _{p \in \ZZ^g}  {\bf e} \left(  \ ^t (p+a) \tau (p+a) + ^t (p+a) (2z)\right)$$
  where $a = \frac{1}{2} \mu$, and $\mu \in \ZZ^g / 2 \ZZ^g$, if we add to $z$ a half-period $\frac{1}{2} ( \e + \tau \e')$, $\e, \e' \in (\ZZ/ 2\ZZ)^g$, we obtain:
  $$ \theta_{\mu} \left( z +  \frac{1}{2}  \tau \e', \tau \right) =    \sum _{p \in \ZZ^g}  {\bf e} \left(  \ ^t (p+a) \tau (p+a) + ^t (p+a) (2z + \tau \e') \right)=$$
  $$= {\bf e} \left( - \frac{1}{4} \ ^t \e' \tau \e' - ^t \e' z\right) \cdot  \sum _{p \in \ZZ^g}  {\bf e} \left(  \ ^t (p+ \frac{1}{2}   (\mu + \e')) \tau (p+\frac{1}{2}   (\mu + \e') ) + ^t (p + \frac{1}{2}   (\mu + \e')) (2z ) \right)=$$
  $$= {\bf e} \left( - \frac{1}{4} \ ^t \e' \tau \e' - ^t \e' z\right) \cdot  \theta_{\mu + \e' } (z , \tau).$$
  
  Whereas $$ \theta_{\mu} \left(z +  \frac{1}{2}   \e, \tau \right) =    \sum _{p \in \ZZ^g}  {\bf e} \left(  \ ^t (p+a) \tau (p+a) + ^t (p+a) (2z +  \e) \right)=$$
  $$ = {\bf e} \left( \frac{1}{2} \ ^t \e  \mu \right)  \theta_{\mu} (z , \tau)= \pm \theta_{\mu} (z , \tau).$$
  
  \begin{cor}\label{thetanull}
  The image points $P_{\eta}$ of the $2$-torsion points $\eta$ of $A$, which are the singular points of $K(A)$,
  are obtained from the point $P_0 = [ \theta_{\mu}(0, \tau)]$ (whose coordinates are called the Thetanullwerte),
  via the projective action of the groups  $(\ZZ/ 2\ZZ)^g$ acting by diagonal multiplication via $ {\bf e} ( \frac{1}{2} ^t \e  \mu)$,
  and $(\tau \ZZ^g / 2 \tau \ZZ^g)$, sending the point $[ \theta_{\mu  } (z , \tau)]$ to $[ \theta_{\mu + \e' } (z , \tau)]$.
  In particular, for $g=2$, letting $P_0 = [a_{00}, a_{10}, a_{01}, a_{11}]$, the sixteen nodes of $K(a)$
  are the orbit of $P_0$ for the action of the group $\sK \cong (\ZZ/2)^4$ which acts via double transpositions 
  and via multiplication with diagonal matrices having  entries in $\{\pm 1\}$ and    determinant $= 1$.
  
  \end{cor}
  \begin{rem}\label{thetanull}
  (I) Observe that the formulae  given by Weber on page 352 of  \cite{weber}   for the double points do
  not exhibit such a symmetry.
  \smallskip
  
  (II) Observe that the orbit of $P_0$ for the  subgroup of diagonal matrices having  entries in $\{\pm 1\}$ and    determinant $= 1$
  consists of $4$ linearly dependent points if $a_{00} a_{10} a_{01} a_{11} = 0$, that is, one of the Thetanullwerte
  vanishes  (this implies, as we shall see,
   that the Abelian surface $A$ is isogenous to a product of elliptic curves) .
  \end{rem}

  \section{Segre's construction and special Kummer quartics}
 
 Beniamino Segre \cite{bsegre} introduced a method in order to construct surfaces with many nodes: take polynomials of the form 
 $$ F(z_1,z_2,z_3,z_4) := \Phi  (z_1^2,z_2^2,z_3^2,z_4^2).$$

Since $\frac{\partial F}{\partial z_i} = 2 z_i \frac{\partial \Phi}{\partial y_i}$, the singular points are the
inverse images of the singular points of $ Y : = \{ y | \Phi (y) = 0\}$ not lying on the coordinate
tetrahedron $T : =  \{ y | y_1 y_2 y_3 y_4 = 0\}$, or of the points where $y_j = 0$ exactly for $j \in J$,
and $\frac{\partial F}{\partial y_i} =0$ for $ i \notin J$ (geometrically, these are the vertices of $T$ lying in $Y$, 
or the points where faces, respectively the edges of $T$ are tangent to $Y$.

The first interesting case is the one where $Y$ is a smooth quadric $Q$, and we shall assume that the four faces
are tangent in points which do not lie on any edge. Hence in this case $F=0$ defines a quartic $X$ with $16$ different singular points,
which are nodes.

The matrix $\sA$ of the dual quadric $Q^{\vee}$ has therefore diagonal entries equal to $0$,
and the condition that no edge is tangent is equivalent to the condition that all nondiagonal entries
are nonzero.

The branch locus of $ X \ra Q \cong (\PP^1 \times \PP^1)$ consists of $T \cap Q$, and in $\PP^1 \times \PP^1$
it gives $4$ horizontal plus $4$ vertical lines. Since we know that double transpositions do not change the cross-ratio
of $4$ points, we can assume that the quadric $Q$ is invariant for the Klein group $\sK'$ in $\PP^3$ generated 
by double transpositions. In short, we may assume that the matrix $\sA$ has the form

\[
		\sA : = \left( 
			\begin{array}{cccc}
				0 & b_2 & b_3 & b_4\\
				 b_2  &0  & b_4 & b_3\\
				 b_3  & b_4  & 0 & b_2 \\
				 b_4& b_3 &  b_2 & 0 \\
				
			\end{array}
		\right),
	\]

where $b_i \neq 0$ $\forall i$ and, using the $\CC^*$ - action  ($K^*$-action if we work more generally with an algebraically closed field $K$ of characteristic $\neq 2$), we may also assume that $b_2=1$.

\begin{rem}
Set $b_i = a_i^2$, $i=2,3,4$. Then the $16$ nodes of $X$ are the points whose coordinates
are the square roots of the coordinates of a column of $\sA$. 

That is, the $4$ points $[0, \pm a_2, \pm a_3, \pm a_4]$ and their orbit under the Klein group $\sK' : = (\ZZ/2)^2$.

We get a set $\sN$ such that, if we define $\sN'$ to be the set of planes orthogonal to some point of $\sN$,
then  $\sN, \sN'$
form a configuration of type $(16_6, 16_6)$.

 In fact,  for instance,  the plane $a_2 z_2 + a_3 z_3 + a_4 z_4$
contains exactly the $6$ points of $\sN$ where $z_j =0$ for some $j = 2,3,4$, $z_1 = a_j$,  
and, if $\{j,h,k\} = \{2,3,4\}$, then $z_k = \e a_h, z_h = - \e a_k$ where $\e = \pm 1$.
\end{rem} 

To calculate the equation of $X$ it suffices here to find the matrix of $Q$, which is, up to a multiple,
 the inverse of $\sA$. And since $Q$ is $\sK'$-symmetric, it suffices to compute its first column.
 
 The calculation of the matrix of $Q$ can be done by hand, since
 
 \begin{align*}
		 \sA (e_1 +  e_3) = b_3  (e_1 +  e_3) + (b_2 + b_4) (e_2 + e_4)\\
		\sA  (e_2 + e_4) = (b_2 + b_4)(e_1 +  e_3) + b_3  (e_2 + e_4)\\
		\sA (e_1 -  e_3) =  - b_3 (e_1 -  e_3) + (b_2 - b_4) (e_2 - e_4)\\
		\sA  (e_2 - e_4) = (b_2 - b_4) (e_1 -  e_3) - b_3 (e_2 - e_4)\\
	\end{align*}
and moreover the inverse of 
\[
		C_1 : = \left( 
			\begin{array}{cc}
				 b_3 & b_2 + b_4\\
				 b_2 + b_4& b_3  \\
				
			\end{array}
		\right),
	\]
is  $(b_3^2 - (b_2 + b_4)^2)^{-1} \cdot D_1$, where:
\[
		D_1 : = \left( 
			\begin{array}{cc}
				 b_3 & - (b_2 + b_4)\\
				 - (b_2 + b_4)& b_3  \\
				
			\end{array}
		\right),
	\]
  while the inverse of 
  \[
		C_2 : = \left( 
			\begin{array}{cc}
				 - b_3 & b_2 - b_4\\
				 b_2 - b_4& - b_3  \\
				
			\end{array}
		\right),
	\]
  equals 
    $(b_3^2 - (b_2 - b_4)^2)^{-1} \cdot D_2$, where:
\[
		D_2 : = \left( 
			\begin{array}{cc}
				 - b_3 & b_4 - b_2 \\
				 b_4 - b_2 & - b_3  \\
				
			\end{array}
		\right).
	\]
	
	Now, a straightforward calculation yields the
	following  coefficients for the first column of the matrix of $Q$:
\begin{align*}
		 \al_0 &= 2 b_2 b_3 b_4 \\
		\al_{01} &=  b_2 (b_2^2 - b_3^2 - b_4^2) \\
		\al_{10} &= b_3 ( b_3^2 - b_4^2 - b_2^2) \\
		\al_{11} &= b_4 ( b_4^2 - b_2^2 -  b_3^2 ) \\
	\end{align*}
	hence $X$ has the following equation in Hudson normal form 
$$  \al_0 (\sum_i z_i^4)  + 2  \al_{01} (z_1^2 z_2^2 + z_3^2 z_4^2) + 2 \al_{10} (z_1^2 z_3^2 + z_2^2 z_4^2) + 2 \al_{11} (z_1^2 z_4^2 + z_2^2 z_3^2) = 0.$$

Observe here  that the Cefal\'u quartic corresponds to the special case where $Q^{\vee}$ is $\s_2 (y) = 0$,
hence the case where  all $b_2,b_3, b_4 =1$, $$Q = \{ y | \s_2 (y) - s_2(y) = 0\}, \ X = \{ z| \s_2 (z_i^2) - s_4(z) = 0\}.$$

\bigskip

We want to explain now the geometrical meaning of the family of Segre-type Kummer quartics,
as Kummer quartics whose associated Abelian surface is isogenous to a product of two elliptic curves.

\begin{prop}\label{special}
The Segre-type (special) Kummer quartics are exactly the Kummer surfaces whose associated principally
polarized Abelian surface $A$ is isogenous to a product $A_1 \times A_2$ of elliptic curves,
$$ A = (A_1 \times A_2) / H, \ H \cong (\ZZ/2)^2,$$ and where $H \ra A_i$ is injective for $i=1,2$. 
\end{prop}
\begin{proof}
Choose 
an isomorphism of the quadric $Q$ with $\PP^1 \times \PP^1$, with coordinates $(u_0, u_1), (v_0,v_1)$.

Then the linear forms $x_i$, since the plane section $\{ x_i = 0\}$ is tangent to $Q$, splits as $ x_i = L_i(u) M_i (v)$.

The Galois cover $ X \ra Q$ with group $(\ZZ/2)^3$ yields sections $z_i $ such that $z_i^2 = x_i = L_i(u) M_i (v)$,
hence the divisor of $z_i$ is reducible, and we can write $ z_i = \la_i \mu_i$, where $ \la_i^2 = L_i(u)$ and
$\mu_i^2 = M_i (v)$.

We define $A_1$ as the $(\ZZ/2)^3$-Galois cover of $\PP^1$ branched on $L_1 L_2 L_3 L_4 (u) = 0$,
and similarly $A_2$ as the $(\ZZ/2)^3$-Galois cover of $\PP^1$ branched on $M_1 M_2 M_3 M_4 (v) = 0$.

The sections $\la_i, \mu_i$ are defined on $A_1 \times A_2$, hence $X$ is the quotient of $A_1 \times A_2$
by the action of $(\ZZ/2)^3$ diagonally embedded in $(\ZZ/2)^3 \oplus (\ZZ/2)^3$.

Clearly $A_i \ra \PP^1$ is the quotient by the group of automorphisms of the form $z \mapsto \pm z + \eta$,
with $\eta \in A_i [2]$ (a 2-torsion point). Hence $A =  (A_1 \times A_2) / H$, where $ H \subset A_1 [2] \times A_2 [2]$
is diagonally embedded for some isomorphism between $A_1 [2] $ and $ A_2 [2]$.

The embedding of $X$ into $\PP^3$ is given by the sections $z_i$, which on $A_1 \times A_2$ can be written
as $ z_i = \la_i \mu_i$, hence they are sections of a line bundle of bidegree $(4,4)$.
The principal polarization then corresponds to a line bundle of bidegree $(2,2)$.

$A$ can also be viewed as an \'etale $(\ZZ/2)^2$-Galois covering of $$A_1 \times A_2
\cong (A_1 / A_1 [2] ) \times (A_2 / A_2 [2]).$$

\end{proof}

\begin{rem}
We can now prove ix) of theorem \ref{Cefalu'}.

Choosing an isomorphism of the smooth quadric $Q$ with $\PP^1 \times \PP^1$, the $\mathfrak S_4$ action shows that the union of the $4$ vertical lines and the $4$ horizontal lines are 
$\mathfrak S_4$-invariant. Hence the $4$  first coordinates are fixed by a group of order $12$, and they form an equianharmonic
cross-ratio. Therefore $A_1$ admits such group of order $12$ permuting the
set of torsion points $A_1[2]$, hence $A_1$  is the equianharmonic = Fermat elliptic curve. Same argument for $A_2$.
\end{rem}
\subsection{Proof of Theorem \ref{real}}

\begin{proof}
Take a Segre-type Kummer quartic corresponding to the choice of $a_2, a_3, a_4 \in \RR$,
so that $X$ is defined over $\RR$ and its singular points are real points.

Observe then  that $b_i > 0$, and that the determinant of $\sA$ is a polynomial in $b_4$ with leading term $b_4^4$.

Hence, for $ b _4 > > 0$, $det (\sA) >0$, as well as the determinant of the inverse. Since $\sA$ is not definite,
its signature is of type $(2,2)$, the same holds for $Q$, so that $Q$ is of hyperbolic type and we may choose 
an isomorphism of the quadric $Q$ with $\PP^1 \times \PP^1$, with coordinates $(u_0, u_1), (v_0,v_1)$,
defined over $\RR$. And to the real singular points correspond real points in $\PP^1 \times \PP^1$.

We have seen that the  Galois cover $ X \ra Q$ with group $(\ZZ/2)^3$ yields sections $z_i $ such that $z_i^2 = x_i = L_i(u) M_i (v)$,
 the divisor of $z_i$ is reducible,  $ z_i = \la_i \mu_i$, where $ \la_i^2 = L_i(u)$ and
$\mu_i^2 = M_i (v)$.
Since $z_i$ is real, and the two terms of its irreducible decomposition are not exchanged by complex conjugation, 
 it follows
that also $\la_i, \mu_i$ are real.

$X$ is embedded in $\PP^3$ by the divisor $H_X$ whose double is the pull back of the hyperplane divisor $H_Q$ of $Q$.

$ Q \cong \PP^1 \times \PP^1$ is embedded in $\PP^3$ by the homogeneous polynomials of bidegree $(1,1)$,
in other words the divisor $H_Q = H_1 + H_2$, where the two effective divisors $H_1, H_2$ are the respective pull backs of the
divisors of degree $1$ on $\PP^1$. Pulling back to $X$, we see that $H_X =  H'_1 +  H'_2$, as we argued in the proof of proposition
\ref{special}.
The previous observation on the reality of the $\la_i, \mu_i$'s shows that both $H'_1, H'_2$ are real.

$Q$ can be embedded in $\PP^{2(m+1)-1}$ by the homogeneous polynomials of bidegree $(1,m)$, as a
variety of degree $2m$. This embedding is given by a vector of rational functions  $\Psi (y_0, y_1, y_2, y_3)$,
ad corresponds to the linear system $| H_1 + m H_2|$.

$X$ is a finite cover of $Q$, hence the pull back of the hyperplane $H_m$ in $\PP^{2(m+1)-1}$
yields the double of a polarization of degree $ 4m$, namely the divisor $(H'_1 + m H'_2)$.
 Hence we have shown that $H'_1 + m H'_2$ is an ample divisor.

To show that $(H'_1 + m H'_2)$ is very ample we first observe that since $|H'_1 + m H'_2| \supset |H'_1 +  H'_2| + (m-1)H'_2$
the linear system yields a birational map. Then we  can invoke for instance  theorem 6.1 of Saint Donat 's article \cite{sd}
on linear systems on K3 surfaces.

The image of $X$ under the linear system $|H'_1 + m H'_2|$, which is associated to a real divisor,  is the desired surface $X'$ for $ d = 4m$, and all the nodes are defined over $\RR$, as well as $X'$.

Projecting $X'$ from one node, and using the theorem of Saint Donat \cite{sd}, as in \cite{nodalsurfaces}, we get the desired surface $X''$
for degree $d = 4m -2$.

The proof for the case where $\RR$ is replaced by an algebraically closed field of characteristic $\neq2$ is
identical, we can  find $a_2, a_3, a_4$ such that $ det (\sA) \neq 0$ because the determinant is a polynomial in the 
$ a_i$ 's 
with leading term $a_4^8$.

\end{proof}

\section{Proof of Theorem \ref{Cefalu'}}
 In this section we want to show how most of the statements, even if proven in other sections of the paper in greater generality, 
 can be proven by direct and elementary computation.
 
 First of all, we observe the necessity of the condition $char(K)\neq 2,3$: because in both  cases the equation $F$ is a square.
 In the other direction, the determinant of the quadric dual to $Q$, which is $\{\s_2(x)=0\} $, equals $-3$, hence $Q$ is smooth if the characteristic is $\neq 2,3$.
 
 ii) ,  after the above remarks, is a straighforward computation (done in the section on special Kummer quartics in greater generality)  based,   if $F (z) = Q(z_i^2)$, 
 on the formula $\frac{\partial F}{\partial z_i} =  2 z_i \frac{\partial Q}{\partial x_i}(z_i^2)$ which implies that some $z_i = 0$.
 
 iii) is based on the fact that $X$ admits $G$-symmetry, and the plane $H_1 := \{z_2 + z_3 + z_4 = 0\}$ intersects $X$
 in an irreducible conic $C_1 : = \{ -z_1 ^2 + z_2^2 + z_3^2 + z_2 z_3 =0 \}$ counted with multiplicity $2$. 
 
This can also be found by direct computation, however
 it follows since $H_1 \cap X$ contains $6$ singular points, no three of which are collinear,
 hence this intersection yields a conic counted with multiplicity $2$.
 
 Therefore we have $16$ curves mapping to $16$ singular points of the dual, which are exactly the same set $\sN$ of nodes of $X$. Since $X^{\vee}$ is also $G$-symmetric, it must equal $X$.
 
 iv) is straightforward using $G$-symmetry.
 
 v) $\ga$ has order $2$ since the Gauss map of $X^{\vee}$ is the inverse of the Gauss map of $X$,
 and $\ga$ is a morphism  on the minimal model $S$ of $X$.
 
 If $A$ is the matrix associated to a projective automorphism of $X$, and we treat $\ga$ as a rational self map of $X$,
$ ^t A \ga A = \ga$, hence $\ga A \ga = ^t A^{-1}$, hence $\ga$ commutes with any orthogonal transformation $A$, 
in particular  $\ga$ commutes with $G$.

 vii) since $\ga$ blows up the point $P$ to the curve $C_P$ which is the set theoretical intersection $H_P \cap X$, where $H_P$ is the orthogonal
 plane to $P$, and $P$ does not belong to $H$, follows that any fixed point of $\ga$ should correspond to a smooth point of $X$
 not lying in any hyperplane $H_P$.
 
 A fixed point would satisfy $[z_i] = [z_i \frac{\partial Q}{\partial x_i}(z_i^2)]$. If all $z_i \neq 0$, then we would have that the point $[1,1,1,1]$ lies in the dual quadric of $Q$, a contradiction since  this is the quadric $\s_2 (x)=0$. 
 
 No coordinate point lies on $X$, hence exactly one or two coordinates $z_i$ must be equal to $0$.
 In the former case we may assume by symmetry that $z_4 =0$, and $z_i \neq 0$ for $i=1,2,3$.
 
 This leads to the conclusion that $(1,1,1)$ must be in the image of the matrix
 
 \[
		 \left( 
			\begin{array}{ccc}
				-2 & 1& 1 \\
				1 & -2& 1 \\
				1 & 1& -2 \\

			\end{array}
		\right),
	\]

 absurd since the sum of the rows equals zero.
 
 If instead $z_3 = z_4 = 0$, $z_1 , z_2 \neq 0$, then we get $x_3 = x_4 = 0, x_1, x_2 \neq 0$.
 
 Here must be that   $(1,1)$ must be in the image of the matrix
 
 \[
		 \left( 
			\begin{array}{cc}
				-2 & 1 \\
				1 & -2 \\

			\end{array}
		\right),
	\]

 hence that $ - 2 x_1 + x_2 = x_1 - 2 x_2 \Leftrightarrow x_1 = x_2 =1$.
 
But $[1,1,0,0]$ is not a point of the quadric $Q$. 
 
 ix) was proven in the section on special Kummer quartics, while x) follows from the calculations we already made.

 We turn now to the assertions i), vi), viii) concerning automorphisms.
 
 i): that $G$ is contained in the group $\sG$ of projective automorphisms of $X$ is obvious, the question is
 whether equality $G = \sG$ holds.
 \begin{rem}
$\sG$ induces  a permutation of  the singular set $\sN$, and of  the set $\sN'$ of the $16$ tropes, the planes dual to the points
 of $\sN$. 
 Given a projectivity $h : X \ra X$, acting with the subgroup $G$ we can assume
 that it fixes the point $P_1 = [1,1,1,0]$ and since the stabilizer of $P_1$ in $G$ acts transitively on the $6$ tropes
 passing through $P_1$,  we may further assume that $h$ stabilizes the plane $\pi : \{ z_1 - z_2 + z_4 = 0\}$.
 
 To determine these automorphisms $h$ (among these  an element of $G$ of order $2$), we need more computations,
 so we shall finish the proof of i) at the end of the section.  
 \end{rem}
 
vi) Observe  that since $\ga$ acts freely on $S$, the quotient $Z : = S / \ga$ is an Enriques surface, therefore $\ga$
  acts as multiplication by $-1$ on a nowhere vanishing regular $2$-form $\omega$. The action of $G$ on  $\omega$
  factors through the Abelianization of $G$, which is the Abelianization $\cong \ZZ/2$ of $\mathfrak S_4$,
  since in the semidirect product, for $\e \in  (\ZZ / 2)^3, \sigma \in \mathfrak S_4$, $\s \e \s^{-1} = \s (\e)$,
  and there are no $\mathfrak S_4$-coinvariants.
  
  The conclusion is that for each $g \in G$ there is a unique element in the coset $ g \{ 1, \ga\}$ which acts trivially on $\omega$,
  hence $G^s \cong G$.
  
   The last assertion is proven in Proposition \ref{fermat}.
  
  viii) The group of birational automorphisms of $X$ coincides with the group of biregular automorphisms of $S$.
  Among the former are the involutions determined by projection from a node $P$. Since however the group $G$
  acts transitively on the set $\sN$ of nodes, two such are conjugate by an element of $G$, hence it suffices to consider just one such involution $\iota$.

  We shall show that $\iota$ does not lie in the group generated  by
  $\ga$ and the projectivities just   looking at its action on the Picard group.
  
  Consider $P_1 = [1,1,1,0]$, so that projection from $P_1$ is given by $w_2 : = z_2 - z_1 , w_3 := z_1 - z_3   , w_4 = z_4$.
  
  We can write the equation $F(z)$ as
  $$ z_1^2 \phi(w) + 2 z_1 \psi (w) + f (w), $$
  where 
  
$$ \phi(w) = 4 (w_2 - w_3)^2 + 6 s_2 (w) - 18 (w_2^2 + w_3^2)= -2 (2 w_2^2 + 2 w_3^2 + 2  (w_2 + w_3)^2 -3w_4^2) $$
		$$ \psi(w) = 2 (w_2 - w_3) s_2 (w) - 6 ( w_2^3 + w_3^3) $$
		$$ f(w) = s_2(w)^4 - 3 s_4(w) ,$$

hence the branch curve (the union of the $6$ lines corresponding to the tropes through $P_1$) has equation
$$ \psi (w)^2 - \phi(w) f(w) =  (w_4 ^2 -  w_2 ^2) (w_4 ^2 -  w_3 ^2)(w_4 ^2 -  (w_2 + w_3) ^2) =$$
$$= (w_4 ^2 -  w_2 ^2) (w_4 ^2 -  w_3 ^2)(w_4 ^2 -  w_1 ^2)= 0$$
where we define $w_1$ so that $w_1 + w_2 + w_3 = 0$.

And
$$ \iota (w_2, w_3, w_4, z_1) = ( \phi (w) z_1 w_2, \phi (w) z_1 w_3, \phi (w) z_1 w_4, f(w)) =$$
$$ = ( \phi (w)  w_2, \phi (w)  w_3, \phi (w)  w_4, - z_1 \phi(w) - 2 \psi(w)).$$
The inverse image of the conic $\{ \phi(w) =0\}$, which is everywhere tangent to the branch locus, 
 splits in $S$ as the exceptional curve $E_1$ plus
a curve $C_1 : = \{ \phi(w) = 0, \ 2 z_1 \psi (w) + f(w) = 0\}$ such that $\iota (C_1)= E_1 $.

 Since the projection is given by the system $ 2 H - 2 E_1$, and $C_1 + E_1$
belongs to this system, the class of $C_1$ equals $2 H - 3E_1$. 

Then $\iota (H) = 3 H - 4 E_1, \iota (E_1) = 2H - 3E_1$.

On the other hand, each projectivity and $\ga$ leave invariant the two blocks given by $\{ E_i | P_i \in \sN\}$,
and $\{ D_i | P_i \in \sN\}$,  so the group generated by $G, \ga, \iota$ is larger than the group $G \times \ZZ/2$
generated by $G, \ga$.

\bigskip

We show now that the group generated by $G, \iota$ is infinite. Indeed, this result holds for all Kummer quartics,
namely we have the following

\begin{theo}\label{infinite}
Let $X$ be a Kummer quartic surface, $S$ its minimal resolution, and assume  that $char(K) \neq 2 $. Then $Aut(S)$ is infinite, and actually the subgroup
generated by the group $G$ of projective automorphisms of $X$, and by an involution $\iota$ obtained via
 projection from one node, is infinite.
\end{theo}

\begin{proof}

Observe in fact that $\iota$ leaves the other nodes fixed, because, if $P_2$ is another node, then the line $\overline{P_1 P_2}$
intersects $X$ only in $P_1, P_2$.

 Hence $\iota$  acts on the Picard group as follows:
$$\iota (H) = 3 H - 4 E_1, \iota (E_1) = 2H - 3E_1, \iota (E_j) = E_j , \ j \neq 1.$$

Let $g \in G$ be a projectivity such that $g(P_1) = P_2, g(P_2) = P_1$. 

Then $\phi: =  g \circ \iota$ acts on the Picard group 
 like this:
$$ \phi (H ) =  3 H - 4 E_2, \  \phi (E_1 ) =   2 H - 3 E_2 , \  \phi (E_2 ) =   E_1,$$
and leaves the subgroup generated by the $E_j$, for $ j \neq 1,2$ invariant.

The subgroup generated by $H, E_1, E_2$ is also invariant and on it $\phi$ acts with the matrix:

 \[
	M : = 	 \left( 
			\begin{array}{ccc}
				3 & 2 & 0\\
				0 & 0 &1 \\
				-4 & -3& 0  \\

			\end{array}
		\right).
	\]
An easy calculation shows that the characteristic polynomial of $M$ is 
$$P_M(\la) = (\la-1)^3,$$
moreover $rank (M- Id) = 2$, therefore the Jordan normal form of $M$ is a single $3 \times 3$ Jordan block with
eigenvalue $1$, and it follows
that $M$ and  $\phi$ have infinite order.

\smallskip
\end{proof}

The proof of viii) is then concluded.

\bigskip

{\bf End of the proof of i).}

Giving an automorphism of $X$ fixing $P_1$ and the plane $z_4 - z_2 + z_1 = 0$, which projects to the line $w_4 - w_2 = 0$,
is equivalent to giving a projectivity of the plane with coordinates $(w_2, w_3,w_4)$ leaving invariant  the line 
$w_4 - w_2 = 0$, and permuting the other $5$ lines of the branch locus. Since all these lines are tangent to the 
conic $\sC : = \{\phi(w) = 0\}$, it is equivalent to give an automorphism of $\sC$ fixing the intersection point
$P' : = \sC \cap \{ w_4 - w_2 = 0\}$, and permuting the other $5$ points.

Projecting $\sC$ with centre $P'$, we see that it is equivalent to  require an affine automorphism permuting these $5$ points.

We note that one such is the permutation $$\tau: w_2 \mapsto w_2, w_4 \mapsto w_4,   w_3 \mapsto w_1 = -w_2 - w_3.$$

So, we calculate: 
$$ P' = (-2,1,-2) \in \{ (-2,1 , \pm 2), (1,1, \pm 2), (1, -2, \pm 2)\}.$$

The projection with centre $P'$ is given by $(w_4 - w_2, w_4 + 2 w_3)$, maps $P'$ to $\infty$,
and the other $5$ points to 
$$ \{ 1,4,0,-2, 2\}.$$
This set is affinely equivalent to 
$$ \sM : =  \{ -3, -1, 0, 1, 3 \},$$
and the above permutation $\tau$ corresponds to multiplication by $-1$.

If $char(K) = 5$, then $\sM = \ZZ/5 = : \FF_5 \subset K$, and
obviously we have then the whole affine group $Aff (1, \FF_5)$, with cardinality $20$, acting on $\sM$.

If instead $char(K) \neq  5$, we observe that the barycentre of $\sM$ is the origin $0$. Assume now that $\al$ is
an affine isomorphism leaving $\sM$ invariant. Then necessarily $\al$ fixes the barycentre, hence $\al$
is a linear map, and $0 \in \sM$ is a fixed point. Hence $\al$ is either of order $4$ or of order $2$,
and we must exclude the first case. But this is easy, since then $\al (1) \in \{-1, 3, -3\}$ should be a square root of
$-1$, and this is only possible (recall here that $2,3 \neq 0$) if $5=0$, a contradiction which ends the proof.

\begin{question}
   Is there a larger finite subgroup of $Aut(S)$ ($S$ the Cefal\'u K3 surface, the minimal model of the Cefal\'u quartic) containing the group isomorphic to $ G \times( \ZZ/2)$
  which is generated by $G, \ga$?
\end{question}

In order to see whether  there are   larger subgroups than $G^s \cong G$ acting symplectically on the Cefal\'u K3 surface, we use first the following:

\begin{lemma}\label{sym-alt}
The group $G$ cannot be a subgroup of the group $AL(2, \FF_4)$ of special affine transformations of
the plane $(\FF_4)^2$, which is a semidirect product 
$$AL(2, \FF_4) \cong (\FF_4)^2 \rtimes \mathfrak A_5 \cong (\ZZ/2)^4 \rtimes \mathfrak A_5 .$$
\end{lemma}
\begin{proof}
Since the order of $G \cong (\ZZ/2)^3 \rtimes \mathfrak S_4$ is $192$, and the order of $ \Ga : = AL(2, \FF_4)$
is $ 920= 5 \cdot 192$,  the existence of  a monomorphism $ G \ra \Ga$ would imply that  the two groups  have
isomorphic $2$- Sylow subgroups $H_G, H_{\Ga}$. But we claim   now  that for the respective
Abelianizations we have
$$ H_G^{ab}  \cong (\ZZ/2)^3, \ \ H_{\Ga}^{ab}  \cong (\ZZ/2)^4.$$

Indeed $H_G = (\ZZ/2)^3 \rtimes H_{\mathfrak S_4}$, where $H_{\mathfrak S_4} \cong D_4$ is generated by a
$4$-cycle $(1,2,3,4)$ and the  double transposition $(1,2)(3,4)$.
The  $4$-cycle $(1,2,3,4)$ acts conjugating all the generators for $ (\ZZ/2)^3$, hence it follows easily 
that the image of  $ (\ZZ/2)^3$ in $H_{G}^{ab}$ is $\ZZ/2$, while the Abelianization of
$H_{\mathfrak S_4} \cong D_4 $ is $  (\ZZ/2)^2$. Hence the first assertion.

We have that $ H_{\Ga} = (\ZZ/2)^4 \rtimes H_{ \mathfrak A_5}$, and $H_{ \mathfrak A_5} \cong (\ZZ/2)^2$
 is  a Klein group of double transposition fixing one element in $\{1,2,3,4,5\} \cong \PP^1 (\FF_4)$.

Hence $H_{ \mathfrak A_5}$ is the group of projectivities $(x_1, x_2) \mapsto (x_1, x_2 + a x_1)$,
for $ a \in \FF_4$. Hence the image of  $ (\ZZ/2)^4$ in $H_{\Ga}^{ab}$ is $\FF_4 \cong (\ZZ/2)^2$,
and $H_{\Ga}^{ab} \cong (\ZZ/2)^4$. 

\end{proof}

\begin{rem}\label{mukai}
1) By using Mukai's theorem 0.6 \cite{mukai} it follows that either $G^s$ is a maximal subgroup acting symplectically,
or it is an index $2$  subgroup of $\Ga$ which acts symplectically, where $\Ga$ is the group of symplectic 
projectivities of the Fermat quartic, 
$$\Ga \cong  (\ZZ/4)^2 \rtimes \mathfrak S_4 \cong  (\ZZ/2)^4 \rtimes \mathfrak S_4,$$
(see \cite{mukai} page 191  for the last isomorphism).

2) The Cefal\'u K3 surface $S$ (as all the other Segre-type Kummer surfaces) possesses two elliptic pencils (induced by the $(\ZZ/2)^3 $ covering $X \ra Q = \PP^1 \times \PP^1$)
and these  are exchanged by the transpositions in 
$\mathfrak S_4$, since for instance the plane $ z_4 =0$ intersects $X$ in two conics:
$$ X \cap \{z_4=0\} = \{ (z_1^2 + \om z_2^2  + \om^2 z_3^2 ) (z_1^2 + \om^2  z_2^2  + \om z_3^2 )= z_4 = 0\}$$
which are exchanged by the transposition $(2,3)$ ($\om$ is here a primitive third root of unity).

It follows then easily that there is a symplectic automorphism exchanging the two pencils, and it suffices
to study the subgroup $G^0 $ of those symplectic automorphisms which do not exchange the two pencils.
By Mukai's theorem and since $G^0 \cap G^s$ has order $96$, the  cardinality of $G^0 $ is either $96$ or $192$.

3) Next, we divide $S$ by the group $(\ZZ/2)^3 \rtimes \mathfrak A_4$, obtaining 
$$Q / \mathfrak A_4 =  (\PP^1 \times \PP^1)/ \mathfrak A_4,$$ 
and set $S' $ to be the quotient by the symplectic index two subgroup
$(\ZZ/2)^2 \rtimes \mathfrak A_4$. 

Letting $S''$ be the Kummer surface double cover of $Q$ branched on the $4$ plane sections
($S''$ is the Kummer surface of $\sE\times \sE$, where $\sE$ is the Fermat elliptic curve),
$S' = S'' / \mathfrak A_4$ and is a double cover of $Q / \mathfrak A_4$.

The second projection of $Q$ onto $\PP^1$ yields an elliptic pencil on $S'$,
corresponding to the projection   $\PP^1 \ra \PP^1 / \mathfrak A_4 \cong  \PP^1$
(and the  fibration $ Q \ra   \PP^1 / \mathfrak A_4$ has  just three singular fibres,
 with multiplicities $(2,3,3)$).

   $G^0 / ( (\ZZ/2)^2 \rtimes \mathfrak A_4)$
is a group $G''$ of cardinality $2$ or $4$, hence  $(\ZZ/2)$ or $(\ZZ/2)^2$.

Let  $G'''$ be the subgroup of $G''$  preserving the fibration and acting  as the identity on the base, hence 
 sending  each fibre to itself. 
 
 Because it is a group of exponent two acting symplectically, it must act on the general fibre via translations of order $2$,
 from this follows   that   $ | G'''| \leq 4$.  But we shall  see next  that $ | G'''| = 4$, hence $G''= G'''$ has cardinality $4$.
 
\end{rem}

\begin{prop}\label{fermat}
The group of symplectic automorphisms of the Cefal\'u K3 surface $S$ has order $384$ and is isomorphic to
the group $\Ga$  of symplectic 
projectivities of the Fermat quartic, 
$$\Ga \cong  (\ZZ/4)^2 \rtimes \mathfrak S_4.$$ 

\end{prop}
\begin{proof}
In view of Mukai's theorem, see remark \ref{mukai}, 1), 2) and 3), we just need to show that 
the Kummer surface $S''$ of $\sE \times \sE$, where $\sE$ is the Fermat elliptic curve,
admits a group of automorphisms preserving the two fibrations induced by  the two product projections of  $\sE \times \sE$
onto  $\sE$, acting sympletically,
and of order $48$, since then $|G^0| = 192$.

For this purpose it suffices to take the group of automorphisms of $\sE \times \sE$
generated by translations by points of order $2$,
and by the automorphism of order $3$ which acts on the uniformizing parameters $(t_1, t_2)$ 
by $$ (t_1, t_2) \mapsto (\om t_1,\om^2  t_2), \ \om^3=1.$$ 

This group descends faithfully to a group of automorphisms of $S'' = Km (\sE \times \sE)$.
And the translations on the first factor  act trivially on the Picard group, hence they extend to the
\'etale covering $ A \ra \sE \times \sE$, where $A$ is the ppav such that $ S = Km (A)$;
therefore  they lift to automorphisms of $S$, thereby  showing that $|G'''| = 4$, as required.

\end{proof}

\begin{remark}\label{Tetrahedron}

We discuss now some claim by Hutchinson, that if $P_1, P_2, P_3, P_4$ are nodes of $X$ such that
the faces $w_i= 0, \ ( i= 1,\dots, 4)$ of the Tetrahedron $T$
having $P_1, P_2, P_3, P_4$ as  vertices are not tangent to $X$, then the Cremona transformation $(w_i) \mapsto (w_i^{-1})$
leaves $X$ invariant.

This is false, let in fact $$e: = \sum_1^4 e_i = (1,1,1,1),  \ P_i : = e - e_i,$$
so that these nodes are an orbit for the  group $\mathfrak S_4$. 

The faces of $T$ are the orbit of $ w_1 : = 2z_1 - z_2 - z_3 - z_4=0$.

Since the point $ (2,-1, -1, -1) \notin X$, as $ 7^2 - 3 (19) \neq 0$, the faces are not tangent to $X$ (we use here $ X = X^{\vee}$).

Consider the node $(0, 1,1, -1)$: the Cremona transformation $(w_i) \mapsto (w_i^{-1})$ sends this node to the point
$$ ( -1, \frac{1}{2}, \frac{1}{2}, - \frac{1}{4}),$$ 
which is not a node of $X$.
\medskip

Hutchinson involutions do of course exist, for instance if we have the Hudson normal form with $\al_0 = 0$, 
$$    \al_{01} (z_1^2 z_2^2 + z_3^2 z_4^2) +  \al_{10} (z_1^2 z_3^2 + z_2^2 z_4^2) +  \al_{11} (z_1^2 z_4^2 + z_2^2 z_3^2) + 2  \be z_1 z_2 z_3 z_4 = 0,$$
the surface is invariant for the Cremona transformation $(z_i) \mapsto (z_i^{-1})$.

More generally, to have a Hutchinson involution, one must take a G\"opel tetrad, that is, the four nodes should 
correspond to an isotropic affine space in the Abelian variety $A$, see \cite{m-o}. 

But we saw in remark \ref{thetanull} that a G\"opel tetrad does  not consist 
of  four independent points if $a_{00} a_{10} a_{01} a_{11}=0$,
that is, in classical terminology, if some Thetanull is vanishing.
\end{remark}

\section{Algebraic proofs for theorems \ref{166} and \ref{gauss}}

We gave short self-contained proofs over the complex numbers, but indeed the theorems hold also for an algebraically closed field 
of characteristic $p \neq 2$. And the results follow easily from existing literature.

The first  result, which I found and explained to Maria Gonzalez-Dorrego, inspired by  the  talk she gave in  Pisa,
and  is contained in \cite{dorrego}, is that there are exactly $3$ isomorphism classes of abstract nondegenerate $(16_6, 16_6)$
configurations, but only one can be realized in $\PP^3$, because of the theorem of Desargues.

The second important result of  \cite{dorrego} is that all these configurations in the allowed combinatorial
isomorphism class are projectively equivalent to the ones obtained as the $\sK$ orbit $\sN$ of a point 
$[a_1,a_2,a_3,a_4]$ satisfying the inequalities
(I), (II), (III), and by the set $\sN'$ of planes orthogonal to points of $\sN$.
 These inequalities are necessary and sufficient to  ensure that the cardinality of $\sN$ is exactly $16$. 
 
 After that,  it is known that two Kummer quartics with the same singular points are the same (one can either see this via projection from a node, since there is only one conic in the plane tangent to the six projected lines,  or another proof can be found in \cite{dorrego}, lemma 2.19, page 67).
 
 Hence it suffices to find explicitly such a Kummer quartic passing through these $16$ points, and this can be done finding
 such a surface in the family of  quartics in Hudson's normal form: the result follows  via the solution of a  system of linear equations
 (mentioned in the statement of theorem \ref{166}).
 
 After these calculations, the proof of the first assertion of theorem \ref{gauss} follows right away since the singular points of $X^{\vee} $ are exactly the ones of $X$,
 hence, by the previous observation, $X = X^{\vee} $.
 
 To prove that the Gauss map induces a fixpoint free involution $\ga$ on the minimal resolution $S$,   instead of carrying out a complicated calculation, 
 we use the easy computations used in the case of the Cefal\'u quartic, which show the assertion for one surface in the family
 (hence  the assertion holds on a Zariski open subset of the parameter space, which is contained in $\PP^3$).
 
 We   show now the result for all Kummer quartics.
 
 In fact, since $\ga$ is an involution, it acts on the nowhere vanishing $2$-form $\omega$ by multiplication by $\pm1$.
 Since for a surface of the family it acts multiplying by $-1$, the same holds for the whole family since the base is connected.
 This implies that $\ga$ can never have an isolated fixed point $x$, because then for local parameters $u_1, u_2$ at $x$ 
 we would have $u_i \mapsto - u_i$
 and $ \omega = (d u_1 \wedge d u_2 + \ {\rm higher \  order \  terms})  \mapsto \omega$, a contradiction.
 
 Exactly as shown  in the case of the Cefal\'u quartic, there are no fixed points on the union of the nodal curves $E_i$
 and of the tropes $D_i$. 
 
  Hence an irreducible  curve $C$ of fixpoints (therefore smooth) necessarily does not intersect the curves $E_i$, nor the tropes $D_i$.
  Hence $ C \cdot E_i = 0$ and  $ C \cdot D_i = 0$ for all $i$;  since $2D_i \equiv H - \sum_{P_j \in P_i^{\perp}} E_j$,
  we infer that $C \cdot H = 0.$ This implies that $C$ is one of the $E_i$'s, a contradiction.
  
  \begin{rem}
  An  argument  similar to the one we just gave, put together with theorem \ref{real}, should allow to recover  the full result of \cite{keum1}
  also for polarized Kummer surfaces  of degree $d = 4m  >  4$ in positive characteristic; the only point to verify is that the involution $\ga$, which is defined on the subfamily corresponding to Segre-type  Kummer surfaces, extends to the whole $3$-dimensional family of polarized
  Kummer surfaces.
  \end{rem}
\section{The role of the Segre cubic hypersurface} 
 
An important role in the theory of Kummer quartic surfaces is played by 
 the 10-nodal Segre cubic hypersurface in $\PP^4$ of equations (in $\PP^5$)
 $$ s_1 (x) : = s_1 (x_0, \dots, x_5) = 0, \ \  s_3 (x)  = 0.$$
 
 It has a manifest $\mathfrak S_6$-symmetry, it contains $10$ singular points, and
 $15$ linear subspaces of dimension $2$, as we shall now see. 

Since  surfaces of  degree $d=4$ with $16$ nodes contain a  weakly even set of cardinality $6$ (see \cite{babbage}, \cite{nodalsurfaces}),  they can be obtained as the discriminant of the projection of
a $10$-nodal  cubic hypersurface $ X \subset \PP^4$ from a smooth point.

It is known  (see \cite{goryunov}) that the maximal number of nodes that a  complex cubic hypersurface $X \subset \PP^m (\CC)$ 
can have equals $\ga(m)$, and $\ga(3) = 4, \ga(4) = 10, \ga(5)= 15, \ga(6) = 35$
(we have in general $\ga(m) = $ $m+1 \choose{[m/2]}$).

Equality is attained, for $m = 2h$ by the Segre cubic
$$ \Sigma (m-1)  : = \{ x \in \PP^{m+1} | \s_1 (x) =  \s_3(x) = 0\} =  \{ x \in \PP^{m+1} | s_1 (x) =  s_3(x) = 0\} ,$$
where as usual $\s_i$ is the i-th elementary symmetric function, and $s_i = \sum_j x_j^i$ is the i-th Newton function.

Whereas, for $ m = 2h +1$ odd, Goryunov produces
$$ T (m-1) : = \{ (x_0, \dots, x_m, z)| \s_1(x) =0, \s_3(x) +  z \s_2(x) +   \frac{1}{12} h (h+1)(h+2) z^3 = 0 \}=$$
$$ =  \{ (x_0, \dots, x_m, z)| s_1(x) =0, 2 s_3(x) - 3 z \s_2(x) + \frac{1}{2} h (h+1)(h+2) z^3 = 0 \}.$$

The nodes of the Segre cubic are easily seen to be the $\mathfrak S_{2h+2}$-orbit of the point $x$ satisfying $x_i=1 , \ i = 0, \dots, h, \ x_j = -1 , \ j =h+1, \dots, 2h+2$.

Whereas the linear subspaces of maximal dimension contained in the Segre cubic hypersurface are 
the $\mathfrak S_{2h+2}$-orbit of the subspace $x_i  + x_{i + h + 1} = 0, \ i = 0,1, \dots, h$.

These are, in the case $h=2$, exactly $15$ $\PP^2$'s.

We have the following result, which is due to Corrado Segre \cite{segre},
and we give a simple  argument here based on an easy  result of \cite{nodalsurfaces}.

\begin{theo}
Any nodal maximizing cubic hypersurface $X$  in $\PP^4$ is projectively equivalent to the Segre cubic.

\end{theo}
\begin{proof}

Let $P$ be a node of $X \subset \PP^4$, and consider the Taylor development of its equation at the point $P = (0, 0, 0, 0, 1)$,
$ F(x,z) = z Q(x) + G(x)$.

Then  by corollary 89 of \cite{nodalsurfaces}, section on cubic fourfolds with many nodes,
the curve in $\PP^3$ given by $Q(x) = G(x) = 0$ is a nodal curve contained in a smooth quadric and has $\ga -1$ nodes,
where $\ga$ is the number of nodes of $X$. But a curve of bidegree $(3,3)$ on $Q = \PP^1 \times \PP^1$
has at most $9$ nodes, equality holding if and only if it consists of three vertical and three horizontal lines.

 Hence $ \ga \leq 9 + 1=10$, and the equality case is projectively unique.

 \end{proof}
 
 It is an open question whether a similar result holds for all Segre cubic hypersurfaces 
 (Coughlan and  Frapporti  in \cite{c-f} proved the weaker result that the small equisingular deformations 
 of the Segre cubic are  projectively equivalent to it).

 It follows that all Kummer quartic surfaces $X$ are obtained as discriminants of the projection of
the Segre   cubic hypersurface $ \Sigma \subset \PP^4$ from a smooth point $P$, which does not lie in
any of the $15$ subspaces $L \subset \Sigma, \ L \cong \PP^2$. 

Of the $16$ nodes of $X$, $10$ are images of the $10$ nodes of $\Sigma $, while $6$ nodes correspond to the
complete intersection of the linear, quadratic, and cubic  terms of the Taylor expansion at $P$ of
the equation of $\Sigma $ (see \cite{nodalsurfaces}, section 9 on discriminants of cubic hypersurfaces).

We point out here an interesting connection to the work of  Coble, (\cite{coble}, page 141) and of van der Geer \cite{vanderGeer}, who proved that 
a compactification of 
the Siegel modular threefold,   the  moduli space for principally polarized Abelian surfaces with
a level $2$ structure,  is the dual variety $\sS$ of the Segre cubic $\Sigma$,
$$\Sigma^{\vee} =  \sS : =  \{ x \in \PP^{5} | s_1 (x) =  s_2(x)^2 -    4 s_4 (x)  = 0\} $$
The singular set of $\sS$,  also called the Igusa quartic, or Castelnuovo-Richmond quartic,
 consists exactly of $15$ lines, dual to the subspaces $L \subset \Sigma, \ L \cong \PP^2$.

 Coble and Van der Geer show that to a general point $P'$ of $\sS$ corresponds the Kummer surface $X'$ obtained by
intersecting $\sS$ with the tangent hyperplane to $\sS$ at $P'$.

Observe that the Hudson equation
$$\al_{00} ^3 -   \al_{00}  ( \al_{10} ^2 + \al_{01} ^2 + \al_{11} ^2 - \be^2) + 2   \al_{10}  \al_{01}  \al_{11} = 0$$ 
defines a cubic hypersurface with $10$ double points as singularities, hence projectively equivalent to
the Segre cubic (\cite{dolgachev}, \cite{hudson}, \cite{d-o}, \cite{vanderGeer}). And to a point of this hypersurface
corresponds the Kummer surface in Hudson's equation above,  intersection of  $\sS$ with the tangent hyperplane
to $\sS$ at the dual point.

\begin{prop}
Let $P \in \Sigma$ be a smooth point of the Segre cubic which does not lie in any of the $15$ planes contained in $\Sigma$.

Then the Kummer quartic $X_P$ obtained as the discriminant for the projection $\pi_P : \Sigma \ra \PP^3$
is the dual of the Kummer quartic $$X'_P: = \sS \cap T_{P'} \sS \subset T_{P'} \sS \cong \PP^3,$$ where $P' \in \sS$ is the dual point of $P$.
\end{prop}
\begin{proof}
By biduality, the  points $y$ of $X'_P$ correspond to the hyperplanes tangent to $\Sigma$ and passing through $P$.
If $z \in \Sigma$ is the tangency point of the hyperplane $y$, then the line $P * z$ is tangent to $\Sigma$,
 hence
 $z$ maps to a point $ x \in X_P$.  
To the hyperplane $y$ corresponds a hyperplane $w$ in $\PP^3$ which contains $x$.

Hence we have to show that, if $ x \in X_P$ is general, then the inverse image of the tangent plane $w$ to $X_P$ at $x$ is the tangent
hyperplane  $y$ at the point $z \in \Sigma$ where the line $P * z$ is tangent to $\Sigma$.

To this purpose, take coordinates $(u, x_0, x_1,x_2,x_3) = : (u,x)$, so that $ P = (1,0)$.

Then we may write the equation $F$ of $\Sigma$ as
$$ L(x)  u^2 + 2 Q(x) u + G(x) = 0.$$
$P$ being a smooth point means that $L(x)$ is not identically zero.

The equation of $X : = X_P$ is $ f (x) : = L(x) G(x) - Q(x)^2 = 0$. If $x \in X$, and  $L(x)$ does not vanish at $x$,
there is a unique point $(u,x)$ such that $\frac{\partial F}{\partial u} = 2 (u L(x) + Q(x))=0$.
In this point the other partial derivatives are proportional to the partials of $f$, since
$$\frac{\partial F}{\partial x_i} = u^2 \frac{\partial L}{\partial x_i} + 2 u  \frac{\partial Q}{\partial x_i} +  \frac{\partial G}{\partial x_i} ,$$
which multiplied by $L^2$ and evaluated at $ x \in X$ yields
$$u^2 L^2 \frac{\partial L}{\partial x_i} + 2 u L^2  \frac{\partial Q}{\partial x_i} + L^2  \frac{\partial G}{\partial x_i} =
LG  \frac{\partial L}{\partial x_i} - 2 QL  \frac{\partial Q}{\partial x_i} + L^2  \frac{\partial G}{\partial x_i} = 
  L \frac{\partial f}{\partial x_i} ,$$
  q.e.d. for the main assertion.
  
  Let us  now prove  the first assertion, first of all let us explain why we must take $P$ outside of these $15$ planes. 
Because, if $P \in \Lam$, then there are coordinates $x_2, x_3$ such that the equation $F$ of $\Sigma$
lies in the ideal $(x_2, x_3)$; then the equation $f$ of $X$ lies in the square of the ideal $(x_2, x_3)$, hence
$X$ has a singular line.

If $P$ is a smooth point, it cannot be collinear with two nodes of $\Sigma$ unless it lies in one of the $15$ planes $\Lam$
contained in $\Sigma$. In fact, using the $\mathfrak S_6$-symmetry, we see that any pair of  nodes is in the orbit
of the pair formed by $(1,1,1,-1,-1,-1)$ and $(1,-1,-1,1,1,-1)$, which is contained in  $\Lam : = \{ z | z_i + z_{i+3}=0\}. $

Hence the $10$ nodes of $\Lam$ project to distinct nodes of $X$ (cf. \cite{nodalsurfaces}),
where the forms $L,Q,G$ do not simultaneously vanish.

There is now a last condition required in order that $\{ L(x) = Q(x) = G(x) = 0\}$
consists of 6 distinct points in $\PP^3$ (which are then nodes for $X$).

In fact $\{ L(x) = Q(x) = G(x) = 0\}$ is the equation of the lines passing through $P$. The Fano scheme
$F_1(\Sigma)$ has dimension $2$ (see for instance \cite{nodalsurfaces}), hence if 
$$\sU \subset F_1(\Sigma) \times\Sigma \subset  F_1(\Sigma) \times \PP^4 $$
is the universal family of lines contained in $\Sigma$, the last condition is that $P$ is not in the set $\sB$ of critical values
of the projection $\sU \ra \Sigma$, which contains the $15$ planes $\Lam \subset \Sigma$ ($\sB$ is the set where the fibre
is not smooth of dimension $0$).

 Indeed, it turns out that $\sB$ is the union of  the $15$ planes $\Lam \subset \Sigma$, see \cite{segre1}, \cite{segre}.
 A slick proof  was given in \cite{d-o} page 184: since the tangent hyperplane at a point of $\Sigma$ yields a quartic 
 which can be put in 
 Hudson normal form, hence has $\sK \cong (\ZZ/2)^4$-invariance, the same holds for the dual. And since there are already $10$
 distinct nodes, it follows that $X$ has exactly $16$ nodes, hence $\{ L(x) = Q(x) = G(x) = 0\}$
consists of $6$ distinct points in $\PP^3$.

\end{proof}

\begin{rem} It is  interesting to study  the modular meaning of the family with parameters $(a_i)$. 
Gonzalez-Dorrego \cite{dorrego} showed that the moduli space of such nondegenerate $(16_6, 16_6)$
configurations is the quotient of the above open set in $\PP^3$ for the action of a  subgroup $H$ of $\PP GL (4, K)$, which is a semidirect product
$$ (\ZZ/2)^4 \rtimes \mathfrak S_6, $$ and which is  the normalizer of $\sK \cong (\ZZ/2)^4$.
 The
group $\mathfrak S_6$ appearing in this and in the other representations (the one of \cite{vanderGeer} for instance),   is the group of permutations of the $6$ Weierstrass points
on a curve of genus $2$, 
and it is known (see \cite{kondo} page 590)  that 
$\mathfrak S_6 \cong Sp (4, \ZZ/2)$, so that the extension is classified by the action of $\mathfrak S_6$ on the group
of $2$-torsion points of the Jacobian of the curve.
\end{rem}

\section{Enriques surfaces \'etale quotients of Kummer K3 surfaces}

We have seen that the Gauss map of a quartic Kummer surface (in canonical coordinates) yields
a fixpoint free involution $\ga : S \ra S$ on the K3 surface $S$ which is the minimal resolution of $X$.

$S$ contains the $16$ disjoint exceptional curves $E_i$, for $P_i \in \sN$, such that $E_i^2 = -2$,
($\sN$ is the set of nodes) and  $16$ disjoint  curves $D_i$, for $i \in \sN$, such that $D_i^2 = -2$,
corresponding to the tropes, that is, the planes orthogonal to $P_i$, such that 
$$ H \equiv 2 D_i + \sum_{E_j \cdot D_i = 1} E_j.$$

And $E_j \cdot D_i = 1$ if and only if the point $P_j$ belongs to the trope $P_i^{\perp}$,
equivalently, if the scalar product $P_i \cdot P_j = 0$.

$\ga$ sends $E_i$ to $D_i$, hence, in the quotient Enriques surface $Z : = S / \ga$,
we obtain $16$ curves $E'_i$ with $(E'_i)^2 = -2$ ($E'_i$ is the image of $E_i$)
and they have the following intersection pattern:
$$ E'_i \cdot E'_j = 1 \Leftrightarrow P_i \cdot P_j = 0, { \rm else} \ E'_i \cdot E'_j = 0.$$

Observe that, since the inverse image of $E'_i$ equals $E_i + D_i$, there is no point of $Z$ where three
curves $E'_i$ pass, because the curves $E_i$ are disjoint, likewise the curves $D_i$.

\begin{defin}
We define the {\bf Kummer-Enriques graph} the graph $\Ga_Z$ whose vertices are the points $P_i$,
and an edge connects $P_i$ and $P_j$ if and only if $ P_i \cdot P_j = 0$.

It is the dual graph associated to the configuration of curves $E'_i$ in the Enriques surface $Z$.

\end{defin}

Without loss of generality, to study $\Ga_Z$ it suffices to consider the special case where $\sN$ is
the set of nodes of the Cefal\' u quartic, the orbit of $[0, \pm 1, \pm 1, \pm1 ]$
under the action of the Klein group $\sK'$ acting via double transpositions.
Because $\sN$ is the orbit of the group $\sK \cong (\ZZ/2)^2 \oplus (\ZZ/2)^2 $,
the graph $\Ga_Z$ is a regular graph. The advantage of seeing $\sN$ as the set of nodes
of the Cefal\' u quartic $X$ is that we can use the group $G$ of projectivities leaving $X$ invariant
as group of symmetries of the graph, since $G$ acts through orthogonal transformations.

\begin{prop}\label{ke}
The  Kummer-Enriques graph  $\Ga_Z$ contains $16$ vertices, $48$ edges, each vertex has 
exactly $6$ vertices at distance $1$, $6$ vertices at distance $2$, $3$ vertices at distance $3$.

It contains $32$ triangles and $48$ edges, so that, adding a cell for each triangle, we get
a triangulation of the real $2$-dimensional torus.

In particular, the maximal number of pairwise non neighbouring vertices is $4$.
\end{prop}

\begin{proof}
There are $16$ vertices, and for each vertex there are exactly $6$ vertices at distance one.
Hence each vertex belongs to $6$ edges and there are exactly $48$ edges.

$ G = C_2^3 \rtimes \mathfrak S_4$, and the set of vertices consists of $4$ blocks of cardinality $4$
(where one fixed coordinate is $0$), permuted by $ \mathfrak S_4$, while $C_2^3$ acts transitively
on each block.

Hence, noticing that elements of the same block are not orthogonal, we see that $G$ acts transitively 
on the set of oriented edges, with stabilizer of order $2$.

We show that each edge belongs to exactly $2$ triangles, hence the number of triangles is $\frac{1}{3} (2 \cdot 48 )= 32$.

By transitivity, consider two neighbouring vertices, namely $P: = [1,1,1,0]$ and $P' : = [0,1,-1, 1]$. 
If $P''$ is orthogonal to both, then necessarily either the second or the third coordinate equals to $0$,
and we get the two solutions $P'' =  [1,0,-1,-1]$ or $P'' =  [1,-1,0,-1]$, which are not orthogonal, thereby proving that
each edge is a side of exactly two triangles.

Adding a cell for each triangle, we get
a triangulation of the real $2$-dimensional torus, since we get a 2-manifold whose  Euler number is $ 16- 48 + 32=0$.

Follows in particular that for each vertex $P$  there are exactly $6$ triangles with vertex $P$.

To count the diameter of the graph, start from $P_1 : = [1,1,1,0]$ and observe that $P_1$ is stabilized by $C_2 \times \mathfrak S_3$ and its $6$ neighbours are 
the $ \mathfrak S_3$-orbit of $[0,1,-1, 1]$, where $\mathfrak S_3$ permutes the first three coordinates.

Each neighbour produces $3$ more neighbours, but in total at distance $2$ we have   $6$ vertices,
the remaining ones with exception of the three vertices   which form the $\mathfrak S_3$-orbit of $[1,1,0,1]$,
namely $[1,1,0,1]$, $[1,0,1,1]$, $[0,1,1 ,1]$.

At distance $2$ from $P_1$ we have the $\mathfrak S_3$-orbit $\sS_1$ of $[1,1,-1, 0]$ (three elements) and 
 the $\mathfrak S_3$-orbit $\sS_2$  of $[1,1,0, -1]$ (three elements).
 
 To show that $[1,1,0,1]$ is at distance $3$ from $P_1  = [1,1,1,0]$ we give the path through 
$P_2  = [0,1,-1, 1] $ and $P_3  = [-1,1,1,0]$.

 Hence our assertion is proven.

We try now to determine  each maximal set $\sM$ of pairwise non-neighbouring vertices.

We  can immediately see two  solutions:
$$\sM_1 : = \{ [1,1,1,0], [1,1,0,1] , [1,0,1,1] , [0,1,1 ,1]\}$$
which consists of four vertices which are pairwise at distance $3$.

Or, the elements of the same block, such as 
$$ \sM_2 : = \{ [1,1,1,0], [1,1,-1,0] , [1,-1,1,0] , [-1,1,1 ,0]\},$$
which consists of vertices which are pairwise at distance $2$.

 We want to see whether there are  other solutions,
up to $G$-symmetries. Now, if there are two vertices at distance $3$
we may assume that these are $ P_1 = [1,1,1,0],  P_4 = [0,1,1 ,1]$. $\sM$ cannot contain 
the $6$ neighbours of $P_1$ and the $6$ neighbours of $P_4$, hence $\sM =  \sM_1 $.

If all the vertices are at distance $2$, then we may assume that $\sM$ contains $P_1$, and then 
some of the six vertices  at distance $2$. We observe that these form two $\mathfrak S_3$-orbits
$\sS_1$ and $\sS_2$, and these have the property that two vertices of the same orbit are not neighbouring,
while two  vertices of different  orbits are  neighbouring.

Hence we find  just another solution
$$ \sM_3 : = \{ [1,1,1,0], [1,1,0,-1] , [1,0 ,1,-1] , [0,1,1 ,-1]\},$$
 and we conclude that all the solutions are in the respective $G$-orbits of $\sM_1, \sM_2, \sM_3$.

\end{proof}

\begin{cor}
The Enriques surface $Z = S / \ga$ contains sets of $4$ disjoint $(-2)$ curves, hence $Z$ is in several ways
the minimal resolution of a $4$-nodal Enriques surface. 
\end{cor}

\begin{rem}
 i) One may ask, in view of proposition  \ref{ke},  whether $Z$ cannot be the minimal resolution of a $5$-nodal Enriques surface:
a  set of  $5$ disjoint $(-2)$ curves on $Z$ would produce on the K3 double cover $S$ 
at least two more $(-2)$-curves exchanged by the involution $\ga$.

ii) We can also consider the $32$ vectors orbit of $(1,1,1,0)$ for the action of $G$. Then we can define
another graph, with $32$ vertices and $96$ edges; this is an unramified double covering of the graph $\Ga_Z$,
and is again associated to a triangulation of the $2$-dimensional real torus.

\end{rem}

\section{Remarks on Normal cubic surfaces and on strict selfduality}

In this section we are mainly concerned with the case of the complex ground field $\CC$, and
we make some topological considerations.

It is worthwhile to observe that, already for degree $d=3$, the maximal number $\mu(d)$ of singular points
can be smaller, for  surfaces with isolated singularities,
than the total sum of the Milnor
numbers of the singularities . In fact, if a cubic surface has $4$ singular points exactly,
it is projectively equivalent to the Cayley cubic of equation $\s_3(x) = 0$, where $\s_3$ is the third elementary
symmetric function,
and its singularities are just four nodes.

But there is the cubic surface $ x y z = w^3$ which possesses $3$ singular points of type $A_2$,
hence realizing a sum of the Milnor numbers equal to $ 6$.

\begin{prop}
The cubic surface $X : = \{  x y z = w^3\}$ is strictly self-dual.
\end{prop}

\begin{proof}
We can change coordinates so that  $X : = \{  x y z = \la w^3\}$, where we can choose $\la \neq 0$ arbitrarily.

The dual map is given by  $$\psi (x,y,z,w) = (yz, xz, xy, - 3 \la w^2) = : (a,b,c,d).$$

We want $ \la d^3 = abc $, this requires that
$$ - 27 \la^4  w^6 =  (xyz)^2 \Leftrightarrow - 27 \la^4  w^6 = \la^2 w^6  \Leftarrow - 27 \la^2   = 1.$$
\end{proof}

In \cite{cil-ded} the authors study limits of dual surfaces of smooth  quartic surfaces: and they show that
the  above cubic surface, as well as the Kummer surfaces, appear as limits.

As we discuss in the following example, this normal cubic is also the one with the largest (finite) fundamental group of
the smooth part among the normal cubic surfaces which are not cones over smooth plane cubic curves
(this is easy to show since the degree of a normal Del Pezzo surface is at most $9$).

\begin{ex}(A cubic with maximal Milnor number)\label{3A_2}

This  is the quotient $\PP^2 / (\ZZ/3) $ for the action such that $$(u_0, u_1,u_2) \mapsto (u_0,\e  u_1, \e^2 u_2),$$
where $\e$ is a primitive third root of unity.

The quotient is embedded by $x_i : = u_i^3, \ i=0,1,2 $ and by $x_3 := u_0 u_1u_2$, so that 
$$\PP^2 / (\ZZ/3) \cong Y : = \{ x_0 x_1 x_2 = x_3^3\}.$$

In this case the triple covering is only ramified in the three singular points $x_3=x_i=x_j = 0$
( for $0 \leq i < j \leq 2$), hence we conclude for $Y^* : = Y \setminus Sing (Y)$, that (since $\PP^2$ minus three points is simply connected)
$$\pi_1 (Y^*) \cong \ZZ/3.$$

This cubic surface  has 3 singular points with Milnor number $2$ (locally isomorphic to the singularity
$ x_1 x_2 = x_3^3$), and $3 \cdot 2 = 6$ is bigger than the maximum number of singular points that a normal cubic can have,
which equals to $4$. The surface realizes the maximum for  the total sum of the Milnor
numbers of the singularities.

In fact, for each normal cubic surface $Y$, which is not the cone over a smooth cubic curve, its singularities are rational double points,
hence the sum $m$ of their Milnor numbers of the singular points equals the number of the exceptional $(-2)$-curves 
appearing in the resolution $\tilde{Y}$; since $\tilde{Y}$ is the blow up of the plane in 6 points, the rank of
$H^2(\tilde{Y}, \ZZ)$ is at most 7, hence (these $(-2)$-curves being numerically independent, see \cite{artin}) 
$ m \leq 6$.

\end{ex}

We give now a very short proof of  the following  known theorem (see \cite{adt} for other approaches)

\begin{theo}
If $Y$ is the Cayley cubic, and $Y^*$ its smooth part,  then 
 $$\pi_1 (Y^*) \cong \ZZ/2.$$
\end{theo}
\begin{proof}
The key point is that  $Y$ admits a double covering ramified only in the four singular points (this can also immediately be seen representing
$Y$ as the determinant of a symmetric matrix of linear forms), which is isomorphic to the blow up of $\PP^2$
in three points. 

Indeed, consider the birational action of $\ZZ/2$ on $\PP^2$ given by the Cremona involution 
$$c : (u_i) \mapsto (\frac{1}{u_i}) = ( \frac{u_0 u_1 u_2}{u_i}) .$$

The involution becomes biregular on the blow up $Z$ of $\PP^2$ at the three points of indeterminacy,
the three coordinate points; $Z$ is the so-called Del Pezzo surface of degree $3$,
and the fixed points of the involution are exactly four, the points $(\e_1, \e_2, 1)$, where $ \e_i \in \{ 1, -1\}$.

We have that the  quotient $ Z / c$  is exactly  a cubic $Y$ with four singular points, hence isomorphic to the Cayley cubic.

Indeed, the quotient map $\psi$  is given by the system of $c$-invariant cubics through the coordinate points $e_i$,
hence
$$ \psi (x) =  (x_1 x_2 x_3, x_1 (x_2^2 + x_3^2),   x_2 (x_3^2 + x_1^2), x_3 (x_1^2 + x_2^2)),$$
and the image $Y$ has degree $ 3 = \frac{1}{2} (3^2 - 3 )$.

$ Z \ra Y$ is only ramified at the four fixed  points, 
whose image points are the 4 singular points $(\e_1 \e_2, 2 \e_1 , 2 \e_2, 2)$, whereas $\psi$ is otherwise injective 
on $Z/c$: hence $Y$ is a cubic with 4 singular points.

Since $Z$ minus a finite number of points is simply connected,
follows that $\pi_1 (Y^*) \cong \ZZ/2.$

\end{proof}

\begin{prop}\label{nodalcubic}
If $Y$ is a normal cubic surface, then it cannot contain three  collinear singular points.

If it contains at least four  singular points, then these are linearly independent and $Y$ is projectively equivalent to the
4-nodal Cayley cubic. 

If $Y$ is a nodal cubic with $\nu \leq 3 $  nodes, then $Y^*$ is simply connected.
\end{prop}
\begin{proof}
There cannot be three collinear  singular points  of $Y$, since otherwise we may take coordinates such that these are
the points $e_0, e_1, e_0 + e_1$ and the equation of the cubic is then of the form $ F = x_2 A(x) + x_3 B(x)$,
since the line $\{ x | x_2 =  x_3=0\}$  is contained in $Y$. Vanishing of the partial derivatives 
$ \frac{\partial F}{\partial x_2 }, \frac{\partial F}{\partial x_3 } $ in the three points
imply that $A,B$ also vanish in the three points, hence $A,B$ belong to the ideal $(x_2, x_3)$ and 
$Y$ is singular along the line $x_2, x_3=0$.

Similarly, if $e_0, e_1, e_2$ are singular points of $Y$, then the equation $F$ of $Y$
contains only monomials $x^m$ with $m_0, m_1, m_2 \leq 1$, hence we can write
$ F = c x_0 x_1 x_2  + x_3  B(x)$.  If also $ e_0 + e_1 +e_2$ is a  singular point of $Y$, 
then $c=0$ and $Y$ is reducible. Hence if there are four singular points, then we may assume them to be
$e_0, e_1, e_2, e_3$ and  then the equation $F$ of $Y$
contains only monomials $x^m$ with $m_0, m_1, m_2, m_3 \leq 1$, hence
$F = \sum_i \frac{a_i } {x_i} x_0 x_1 x_2  x_3$, and we may replace $x_i $ by $\la_i x_i$
obtaining that $a_i = 1$.

Therefore if $Y$ is nodal, the $\nu$ nodes of $Y$ are linearly independent, and it follows that the nodal cubics
with  $\nu$ nodes are parametrized by a Zariski open set in a linear space, in particular, for $\nu$
fixed, $Y^*$ has always the same topological type.

We finish observing that a cubic surface with $\nu \leq 3$ nodes  is the blow up of the plane
in six points, of which $\nu$ triples are collinear  (the case of six points on a smoth conic, with $\nu=1$,
reduces to the previous case through a standard Cremona transformation based at $3$ of the six points). In particular, the fundamental group $\pi_1 (Y^*)$
is the quotient of the fundamental group of the complement in $\PP^2$ of at most three general  lines.

But $\PP^2$ minus three general  lines is $(\CC^*)^2$, and $\pi_1 ((\CC^*)^2) = (\ZZ)^2$. Hence 
$\pi_1 (Y^*)$ is Abelian, hence equal to $H_1 (Y^*, \ZZ) $, which is a trivial binary code for $\nu \leq 4$
(see \cite{nodalsurfaces}).

\end{proof} 

 Francesco Russo \cite{russo}  informed us of other examples of strictly self-dual hypersurfaces, the simplest one
being:
\begin{ex} ( Francesco Russo' s generalization of the Perazzo cubic \cite{perazzo})

Consider in $\PP^{2n+1} $ the hypersurface of equation
$$ \{x_0 \dots x_n - y_0 \dots y_n = 0 \}.$$

The dual map is given by:

$$ \psi(x,y) = (u,v), \ u_i = \frac{1}{x_i} x_0 \dots x_n, v_i =-  \frac{1}{y_i}y_0 \dots y_n.$$

For $n$ odd we have $u_0 \dots u_n - v_0 \dots v_n = 0$, so $X$ equals its dual variety.

For $n$ even it suffices to take $$X = \{x_0 \dots x_n + \la   y_0 \dots y_n = 0 \}, \la^2 = -1.$$
\end{ex}

The other examples constructed by Russo in \cite{russo} use determinants, symmetric determinants or Pfaffians,
and the Perazzo trick.

\section{Appendix on monoids.}

We consider now a normal quartic surface $X = \{ F =0\} \subset \PP^3_K$, where $K$ is an algebraically closed field. 

If $X$ has a point of multiplicity $4$, then this is the only singular point, while 
if $X$ contains a triple point, we can write the equation 
$$ F (x_1, x_2,x_3, z) = z G (x) + B (x) ,$$
and, setting $G_i : = \frac{\partial{G}}{\partial{x_i}}, B_i : = \frac{\partial{B}}{\partial{x_i}}$, we have 
$$ Sing(X) = \{ G(x) = B(x) = G_i z + B_i = 0, \ i=1,2,3 \} \ .$$

If $(x,z) \in Sing(X)$ and $x \in  \{ G(x) = B(x) = 0\}$, then $ x \notin Sing (\{G=0\})$, else $ x \in Sing (\{G=0\}) \Rightarrow  x \in Sing (\{B=0\})$ and the whole line $(\la_0 z, \la_1 x) \subset Sing(X)$. Hence $\nabla (G) (x) \neq 0$ and there exists a unique
singular point of $X$ in the above line. Since the two curves  $ \{ G(x) = 0\}, \{ B(x) = 0\}$
have the same tangent at  $x$ the intersection multiplicity at $x$ is at least $2$, and we conclude:

\begin{prop}\label{monoid}
Let $X$ be  quartic surface $X = \{ F =0\} \subset \PP^3_K$, where $K$ is an algebraically closed field,
and suppose  that $ Sing (X)$ is a finite set. If $X$ has a triple point then $ | Sing(X)| \leq 7$.

More generally, if  $X$ is a normal monoid, that is, a degree $d$ normal surface $X = \{ F =0\} \subset \PP^3_K$, where $K$ is an algebraically closed field, possessing a  point of multiplicity $d-1$, then $ | Sing(X)| \leq 1 + \frac{d (d-1)}{2}$.

\end{prop}

\begin{proof}
The second assertion follows by observing that in proving the first we never used the degree $d$, except for
concluding that the total intersection number (with multiplicity) of $G,B$ equals $ d (d-1)$.

\end{proof}

Finally,  I would like to thank Thomas Dedieu, Igor Dolgachev,  Francesco Russo,  for  useful comments on the first version of the paper;  and  Shigeyuki Kondo and  especially  the referee for very stimulating and useful remarks and queries.

\end{document}